\documentclass[11pt,a4paper,reqno]{amsart}
\usepackage[english]{babel}
\usepackage[latin1]{inputenc}
\usepackage[T1]{fontenc}
\usepackage{latexsym,amsfonts,amssymb,amsmath,amsthm,textcomp, mathtools}
\usepackage{graphicx}
\usepackage{tikz}
\usepackage{fancyhdr}
\usepackage{caption}
\usepackage[hidelinks]{hyperref}
\usepackage{calc}
\usepackage{chngcntr}

\newcommand*{\bigboxplus}{\DOTSB\mathop{\mathpalette\big@boxplus\relax}\slimits@}

\newtheorem{thm}{Theorem}
 \newtheorem{prop}[thm]{Proposition}
 \newtheorem{lemma}[thm]{Lemma}
 \newtheorem{cor}[thm]{Corollary}
 \newtheorem{Proposition}[thm]{Proposition}

 \theoremstyle{definition}
 \newtheorem{definition}[thm]{Definition}
 \newtheorem{ex}[thm]{Example}

\newtheorem{const}[thm]{Construction}

 \theoremstyle{remark}
 \newtheorem{remark}[thm]{Remark}
\numberwithin{thm}{section}
\numberwithin{equation}{section}
\theoremstyle{definition}

\def \Q {\mathbb{Q}}
\def \Z {\mathbb{Z}}
\def \R {\mathbb{R}}

\def \B {\mathbb{B}}

\def \X {\mathbb{X}}

\def \Aut{\mathop{\mathrm{Aut}}\nolimits}
\def \Gal{\mathop{\mathrm{Gal}}\nolimits}
\def \ad{\mathop{\mathrm{ad}}\nolimits}
\def \Bun{\mathop{\mathrm{Bun}}\nolimits}
\def \det{\mathop{\mathrm{det}}\nolimits}
\def \deg{\mathop{\mathrm{deg}}\nolimits}

\def \dim{\mathop{\mathrm{dim}}\nolimits}

\def \FilVect{\mathop{\mathrm{FilVect}}\nolimits}

\def \Fil{\mathop{\mathrm{Fil}}\nolimits}

\def \HN {{\rm HN}}
\def \hn{\mathop{\mathrm{HN}}\nolimits}

\def \hom{\mathop{\mathrm{Hom}}\nolimits}

\def \lim{\mathop{\mathrm{lim}}\nolimits}

\def \Newt{\mathop{\mathrm{Newt}}\nolimits}
\def \Id{\mathop{\mathrm{Id}}\nolimits}

\def \Isoc{\mathop{\mathrm{Isoc}}\nolimits}
\def \Rep{\mathop{\mathrm{Rep}}\nolimits}

\def \tri{\mathop{\mathrm{tri}}\nolimits}

\def \Spec{\mathop{\mathrm{Spec}}\nolimits}
\def \Spa{\mathop{\mathrm{Spa}}\nolimits}

\def \rk{\mathop{\mathrm{rk}}\nolimits}
\def \rank{\mathop{\mathrm{rank}}\nolimits}

\def \val{\mathop{\mathrm{val}}\nolimits}
\def \Vect{\mathop{\mathrm{Vect}}\nolimits}
\def\BG{{$B_{{\rm dR}}^+$-Grassmannian }}

\def\GL{{\rm GL}}
\def\Gr{{\rm Gr}}

\def\Gal{{\rm Gal}}
\def\BB{{\rm BB}}

\def\rk{{\rm rk\,}}

\def\pr{{\rm pr}}
\def\ad{{\rm ad}}
\def\av{{\rm av}}

\def\Bun{{\rm Bun}}
\def\wa{{\rm wa}}
\def\a{{\rm a}}
\def\dom{{\rm dom}}

\def\dim{{\rm dim}\,}

\def\Spa{{\rm Spa}}

\def\dl{(\!(}
\def\dr{)\!)}
\def\ll{[\![}
\def\rr{]\!]}

\def\dR{{\rm dR}}

\def\R{{\mathbb R}}

\def\Z{{\mathbb Z}}

\def\Fl{{\mathcal F\ell}}
\def\E{{\mathcal E}}
\def\Q{{\mathbb Q}}

\def\O{{\mathcal O}}

\def\BdR{{B_{{\rm dR}}}}
\def\B+{{B^+_{{\rm dR}}}}

\def\HT{{\rm HT}}
\def\hn{Harder-Narasimhan }
\def\v{{v}}
\def\X{{Z}}
\long\def\forget#1{}

\begin{document}
	\setcounter{tocdepth}{1}
	\selectlanguage{english}
	\title{A Harder-Narasimhan stratification of the $\B+$-Grassmannian}
	\author{Kieu Hieu Nguyen and Eva Viehmann}

\address{Technische Universit\"at M\"unchen\\Fakult\"at f\"ur Mathematik - M11 \\ Boltzmannstr. 3\\85748 Garching bei M\"unchen\\Germany}
\email{nguyenki@ma.tum.de, viehmann@ma.tum.de}%This email???
\thanks{The authors acknowledge support by the ERC in form of Consolidator Grant 770936:\ NewtonStrat and by Deutsche Forschungsgemeinschaft  (DFG) through the Collaborative Research Centre TRR 326 "Geometry and Arithmetic of Uniformized Structures", project number 444845124.}
\begin{abstract}
{We establish a Harder-Narasimhan formalism for modifications of $G$-bundles on the Fargues-Fontaine curve. The semi-stable stratum of the associated stratification of the $\B+$-Grassmannian coincides with the variant of the weakly admissible locus defined Viehmann, and its classical points agree with those of the basic Newton stratum. When restricted to minuscule affine Schubert cells, the stratification corresponds to the Harder-Narasimhan stratification of Dat, Orlik and Rapoport. We also study basic geometric properties of the strata, and the relation to the Hodge-Newton decomposition.}
\end{abstract}
	\maketitle
	
	\tableofcontents
	
\section{Introduction}

Let $p$ be a prime and let $G$ be a connected reductive group over a finite extension $F$ of $\Q_p$. We fix a conjugacy class of a minuscule cocharacter $\mu$ of $G_{\overline F}$, and denote by $E$ its field of definition, a finite extension of $F$. Let $\Fl(G,\mu)$ be the associated flag variety, considered as an adic space over $\breve E$. Here $\breve E$ denotes the completion of the maximal unramified extension of $E$.

Let $b\in G(\breve F)$. In \cite{RZ96}, Rapoport and Zink define the weakly admissible locus $\Fl(G,\mu,b)^{\wa}\subseteq\Fl(G,\mu)$. It is an open adic subspace of $\Fl(G,\mu)$ having an explicit description as the complement of a certain profinite union of linear subspaces. Its points are characterized by a semi-stability condition on the associated filtered isocrystals with additional structure.

Let $\Fl(G,\mu,b)^{\a}$ denote the admissible locus, an open adic subspace of $\Fl(G,\mu)$. The existence of the admissible locus has been conjectured by Rapoport and Zink, see
 \cite[Conj.~11.4.4]{DOR}. It is characterized by being a subspace of  $\Fl(G,\mu,b)^{\wa}$ 
 having the same classical points and such that there exists a $p$-adic local system with additional structure over $\Fl(G,\mu,b)$ that interpolates the crystalline representations attached to all classical points. For $(G,\mu)$ of PEL type, Hartl \cite{HartlAnnals} and Faltings \cite{Falt2010} gave a construction using the Robba ring, resp.~the crystalline period ring. In general, the admissible locus is constructed using modifications of $G$-bundles on the Fargues-Fontaine curve, and based on work of Fargues-Fontaine \cite{FF}, Fargues \cite{FarGtors}, Kedlaya-Liu \cite{KL} and Scholze \cite{SW17}. As was recently shown in \cite{CFS}, the admissible locus and the weakly admissible locus coincide only in exceptional cases.
 
In order to understand the relation between these two subspaces, it is helpful to view them as part of a broader theory. The first generalization -- carried out among others in work of Fargues and Scholze -- is to consider the adic Newton stratification instead of just the admissible locus (which is the unique open Newton stratum). The second step -- and topic of this article -- is then to establish a similar stratification whose unique open stratum is a variant of the weakly admissible locus. Let us begin by recalling the Newton stratification. To explain this, we also switch to the more general context of fixing any $G(\overline F)$-conjugacy class $\{\mu\}$ of cocharacters of $G_{\overline F}$, and consider affine Schubert cells in the $\B+$-Grassmannian instead of flag varieties. 

By $\Gr_G$ we denote the $B_{{\rm dR}}^+$-Grassmannian as in \cite[Def. 19.1.1]{SW17}. Let $C$ be a complete and algebraically closed field extension of $F$. Then the Cartan decomposition yields $$\Gr_G(C)=\coprod_{\{\mu\}}G(B_{\dR}^+(C))\mu(\xi)^{-1}G(B_{\dR}^+(C))/G(B_{\dR}^+(C))$$
where $\xi$ is a uniformizer of $B_{\dR}^+(C)$ and where we take the union over all $G(\overline F)$-conjugacy classes of cocharacters $\mu$ of $G_{\overline F}$. We obtain an induced subdivision $\Gr_G=\coprod \Gr_{G,\mu}$ into locally spacial sub-diamonds called affine Schubert cells, and also the closed affine Schubert cells $\Gr_{G,\leq\mu}=\coprod_{\mu'\leq\mu}\Gr_{G,\mu}$. For minuscule $\mu$, the natural Bialynicki-Birula map $$\BB_{\mu}:\Gr_{G,\mu}\rightarrow \Fl(G,\mu)^{\diamond}$$ is an isomorphism between the affine Schubert cell and the diamond associated with $\Fl(G,\mu)$. In general, the Bialynicki-Birula map exists, but is not an isomorphism.

Let $X$ be the Fargues-Fontaine curve over the tilt $C^{\flat}$ of $C$ and let $\infty \in X$ be the point corresponding to its untilt $C$. By \cite{FarGtors} we have a bijection between the set of isomorphism classes of $G$-bundles on $X$ and the Kottwitz set $B(G)$. Let $\E_b$ be the $G$-bundle corresponding to the fixed element $b$. Using Beauville-Laszlo uniformization we have for every $x\in \Gr(C)$ a $G$-bundle $\E_{b,x}$ on $X$  that is obtained as a modification of $\E_b$ at $\infty$, compare \cite[III.3]{FarguesScholze}. Mapping $x$ to the isomorphism class of $\E_{b,x}$ one obtains a map $$|\Gr_{G}|\rightarrow |\Bun_G|\cong B(G)$$ to the stack of $G$-bundles on the Fargues-Fontaine curve. For $b=1$, the image of $|\Gr_{G,\mu}|$ or $|\Gr_{G,\leq \mu}|$ consists of the finite set $B(G,-\mu)$. For any $[b']\in B(G)$ we denote by $\Gr_{G,\mu,b}^{[b']}$ the locally closed subspace of $\Gr_{G,\mu}$, and also the induced locally spatial sub-diamond, called the Newton stratum for $[b']$ and $G,\mu,b$.

By \cite[Section 4]{Viehmann20} there is also a generalization of the definition of the weakly admissible locus $\Gr_{G,\mu,b}^{\wa}\subset \Gr_{G,\mu}$ based on modifications of $G$-bundles on the Fargues-Fontaine curve. In the minuscule case it can be identified with the classical weakly admissible locus in $\Fl(G,\mu)$ via the Bialynicki-Birula isomorphism. By \cite[Thm.~1.3]{Viehmann20} the weakly admissible locus in a minuscule affine Schubert cell $\Gr_{G,\mu}$ intersects all Newton strata that are Hodge-Newton indecomposable.

In this article, we embed the notion of weak admissibility in a broader theory by constructing a Harder-Narasimhan stratification of $\Gr_{G}$. Intersecting the semi-stable stratum with any $\Gr_{G,\mu}$ one obtains the weakly admissible locus of \cite{Viehmann20}. We use a natural construction in terms of modifications of $G$-bundles on the curve that ressembles the definition of Fargues-Fontaine's Harder-Narasimhan formalism that led to the Newton stratification. The main difference between the two theories is in the choice of the subbundles to define semi-stability. Fargues and Fontaine consider for a given vector bundle $\E$ on $X$ all subbundles $\E'$ and their slopes. We consider vector bundles $\E=\E_{b,x}$ obtained as a modification \`a la Beauville-Laszlo of some $\E_b$, and only those subbundles that correspond to direct summands of $\E_b$. The slope of such a subbundle of $\E_{b,x}$ is then the same as in the other theory.  From this comparison one immediately obtains that the Harder-Narasimhan invariant is bounded above by (the opposite of) the Newton invariant. In particular, this yields the analog of the result that admissible implies weakly admissible.
 
To extend our theory to $G$-bundles for $G$ other than $\GL_n$ we then use the Tannakian formalism and a result of Cornut and Peche Irissarry \cite{CP}. For this we have to restrict to modifications of the trivial bundle $\E_1$, which is, however, also the most relevant case for the consideration of Newton strata. Nevertheless it would be interesting to see if our theory can be generalized to all $G$ and $b$.

In the past years, several closely related theories have been introduced. More precisely, our definition of the Harder-Narasimhan theory (or close variants of it) have been studied by several authors. However, most of the geometric and comparison properties that we establish are new. Let us comment in more detail. 

In \cite[Section 9.3]{Far19}, Fargues proved the existence of a Harder-Narasimhan filtration for pairs $ (\mathcal{E}, x) $ when  $\mathcal{E}$ is the trivial bundle of rank $n$ and $ x \in \Gr_{n,\mu} (C) $ for some minuscule $\mu$. Our results extend this theory to the case where $\mu$ is arbitrary and where $G$ is any reductive group instead of $\GL_n$ or where the bundle $\E$ may be non-trivial. We also discuss the minuscule case of several of Fargues' conjectures on the geometric properties of the Harder-Narasimhan strata that he stated in \cite[Section 9.7]{Far19}. 

Cornut and Peche-Irissarry \cite{CP} generalized Fargues's theory by defining a Harder-Narasimhan theory for Breuil-Kisin-Fargues modules and call the resulting filtration the Fargues filtration. It coincides with the particular case for $b=1$ of our definition, and for general groups, we use their results. 

In \cite{DOR}, Dat, Orlik and Rapoport use filtered isocrystals and define the weakly admissible locus in the flag variety $\Fl(G,\mu)$ also for non-minuscule $\mu$ . They generalize this also to a Harder-Narasimhan stratification of the flag variety where again the weakly admissible locus is the semi-stable and unique open stratum. One can consider the inverse images  of Dat-Orlik-Rapoport's Harder-Narasimhan strata under the Bialynicki-Birula map, which then define a decomposition of the affine Schubert cell $\Gr_{G,\mu}$. This approach is taken by Shen in the first versions of \cite{Shen}. This theory coincides with our Harder-Narasimhan stratification if $\mu$ is minuscule. In general, already the notions of weak admissibility differ, compare \cite[Ex.~4.10]{Viehmann20}. 

When we finished the present manuscript, Shen posted a new version of \cite{Shen} on the arxiv, in which he changed his definition and adapted several of his arguments, so that part of his new version is in parallel to some of our results (the construction of the Harder-Narasimhan theory for $b=1$, as well as the proof that the strata are locally closed). Most of our geometric results of Sections 5-8 are independent of his work.

Despite the fact that the Harder-Narasimhan theory that we study has been considered (at least in some form) by many people, very little was known about the properties of the associated decomposition of $\Gr_G$.

In Proposition \ref{thmnonempty1} we give a group-theoretic criterion for non-emptiness. Recall from \cite{RapoportAppendix} that the set of non-empty Newton strata (for modifications of the trivial bundle) in $\Gr_{G,\mu}$ is indexed by the set $B(G,-\mu)$. In Proposition \ref{thmnonempty1} we show that the set of non-empty Harder-Narasimhan strata in $\Gr_{G,\mu}$ is contained in an -- in general strict -- subset of $B(G,\mu)$ that has an explicit description in terms of the involved Newton points. In \cite{Orlik06}, Orlik proves a non-emptiness criterion for the Dat-Orlik-Rapoport stratification for $G=\GL_n$ (and filtrations of an arbitrary isocrystal). As a byproduct of our theory, we generalize Orlik's description to the case of any $G$, but the trivial $G$-isocrystal in Proposition \ref{propnonemptyDOR}. 

From the definition of the two stratifications, one obtains the immediate estimate that the Harder-Narasimhan polygon is bounded by the Newton polygon. In Section \ref{secnewt} we extend this comparison, showing that in Hodge-Newton decomposable cases, the Hodge-Newton decomposition also is a coarsening of the Harder-Narasimhan reduction. We use this to study cases where Newton strata and Harder-Narasimhan strata coincide.
We continue in Section \ref{secdim} by giving an estimate for the dimension of Harder-Narasimhan strata in the minuscule case. It disproves \cite[Conj.~2 (2)]{Far19}, where Fargues conjectured that the dimension of a Harder-Narasimhan stratum for some $[b']\in B(G,\mu)$ should coincide with the dimension of the Newton stratum for some $[b']^*\in B(G,-\mu)$ where we refer to the next section for the notation. In Proposition \ref{propdim2} we prove that the dimension of a Newton stratum is an upper bound for the dimension of the corresponding Harder-Narasimhan stratum and determine the cases in which equality holds.\\ 

{\it Acknowledgement.} The second author thanks Laurent Fargues and Michael Rapoport for helpful conversations. We thank Miaofen Chen and Ian Gleason for pointing out an error in a previous version.
	
\section*{Notation}
We use the following notation: 
\begin{itemize}
\item $F$ is a finite degree extension of $ \Q_p $ and $\pi$ is a uniformiser of $\mathcal{O}_F$. Let $ \overline{F} $ be an algebraic closure of $ F $ and $ \Gamma = \Gal(\overline{F}|F) $.
\item $ \breve F $ is the completion of the maximal unramified extension of $F$ with Frobenius $ \sigma $. 
\item $ G $ is a connected reductive group over $F$. Let $H$ be a quasi-split inner form of $G$ and fix an inner twisting $ G_{\breve F} \overset{\sim}{\longrightarrow} H_{\breve F} $. 
\item  $ A \subseteq T \subseteq B $ where $ A $ is a maximal split torus, $ T = Z_H(A) $ is the centralizer of $ A $ in $T$ and $ B $ is a Borel subgroup in $H$. Let $ U $ be its unipotent radical.
\item $ (X^*(T), \Phi, X_*(T), \Phi^{\vee}) $ is the absolute root datum of $G$ with positive roots $ \Phi^+ $ and simple roots $ \Delta $ with respect to the choice of $B$.
\item  Further, $ (X^*(A), \Phi_0, X_*(A), \Phi^{\vee}_0) $ denotes the relative root datum, with positive roots $ \Phi^+_0 $ and simple roots $ \Delta_0 $.
\item On $ X_*(A)_{\Q} $ resp.~$ X_*(T)_{\Q} $ we consider the partial order given by $ \nu \leq \nu' $ if $ \nu' - \nu $ is a non-negative rational linear combination of positive resp.~relative coroots.
%\item For $G(??)$-conjugacy classes of cocharacters $\mu$ of $G$ we often consider a dominant representative in $X_*(T)$, also denoted $\mu$. 
\item Let $ C | \overline{F} $ be an algebraically closed complete field. Let $C^{\circ}$ resp. $C^{\flat,\circ}$ be the subring of power-bounded elements of $C$ resp. $C^{\flat}$ and let $\xi$ be a generator of the kernel of the surjective map $W(C^{\flat,\circ})\rightarrow C^{\circ}$. Let $\B+:=\B+(C)$ be the $\xi$-adic completion of $W(C^{\flat,\circ})[1/p]$ and $\BdR=\BdR(C)=\B+[\xi^{-1}]$. Then $\B+\cong C\ll \xi\rr$ and $\BdR\cong C\dl \xi\dr$.
\item Let $X$ be the schematic Fargues-Fontaine curve over $C^{\flat}$. The untilt $C$ of $C^{\flat}$ corresponds to a point $ \infty \in |X|$ with residue field $C$ and $\widehat{\mathcal{O}}_{X, \infty}\cong \B+$.
\item The adic space $Y=\Spa W_{\O_F}(C^{\circ})\setminus \{[\pi]p=0\}$ is equipped with a Frobenius $ \varphi $. The quotient $ Y / \varphi^{\Z} =: X^{\ad} $ is the adic Fargues-Fontaine curve. Recall that there is an equivalence between the categories of coherent sheaves over $X$ and over $X^{\ad}$. 
\item Let $B(G)$ be the set of $G(\breve F)$-$\sigma$-conjugacy classes of elements of $G(\breve F)$. By work of Kottwitz, elements $[b]$ are classified by their Kottwitz point $\kappa_G(b)\in \pi_1(G)_{\Gamma}$ and their Newton point $\nu_b\in X_*(A)_{\mathbb Q,\dom}$.
\item For a $G$-bundle $\E$ on $X$ let $\Newt(\E)$ be the corresponding class in $B(G)$.
\item Let $[b]\in B(G)$. Due to different sign conventions for Harder-Narasimhan polygons and Newton polygons, we write $[b]^*$ for the unique element of $B(G)$ with $\kappa_G([b]^*)=-\kappa_G(b)$ and $\nu_{[b]^*}=(-\nu_b)_{\dom}$. We have $[b]^*\in B(G,-\mu)$ if and only if $[b]\in B(G,\mu)$.
\end{itemize} 	
%\mar{mu as dominant representative of its conj class ony for qsplit G: needed? introduce?}

\section{The Harder-Narasimhan formalism for modifications of vector bundles}	

\subsection{Generalities on vector bundles and gluing}

 Let $ \Bun (X) $ denote the category of vector bundles on the Fargues-Fontaine curve $X$ and recall from \cite{FF} that every object in $ \Bun (X) $ can be written as direct sum of simple objects which can be parametrized by $ \Q$. More precisely, the isomorphism classes of vector bundles of some rank $n$ on $X$ are in bijection with $B(\GL_n)$, the set of $\sigma$-conjugacy classes of elements $b \in \GL_n(\breve F)$, or with isomorphism classes of $\sigma$-isocrystals of rank $n$ over $\breve F$. Here, an isocrystal $ \mathcal{D} = (\breve F^n, \varphi_b := b \sigma ) $ corresponds to the vector bundle
\[
\mathcal{E}_b = Y \times_{\varphi^{\Z}} \mathcal{D} \longrightarrow Y / \varphi^{\Z} = X^{\ad}
\] 
over $X^{\ad}$, and to the corresponding bundle over $X$. Here, the automorphism $\varphi$ on the left hand side acts via $\varphi$ on $Y$ and as $\varphi_b$ on $\mathcal D$.

If $ f : (\breve F^{n_1}, b_1 \sigma ) \longrightarrow (\breve F^{n_2}, b_2 \sigma ) $ is a map of isocrystals, we obtain an induced morphism of vector bundles $ \mathcal{E}(f) : \mathcal{E}_{b_1} \longrightarrow \mathcal{E}_{b_2} $. Be aware that the converse does not hold.

Recall that $X$ is the Fargues-Fontaine curve over $C^{\flat}$, which comes equipped with a point $\infty$ corresponding to $C$. By Beauville-Laszlo's gluing theorem \cite{BL} we have a bijective correspondence between vector bundles $\E$ on $X$ and triples $(\E^e,\E_{\B+},\iota)$ where $\E^e$ is a vector bundle over $X\setminus \{\infty\}$, where $\E_{\B+}$ is a vector bundle on $\Spec(\B+)$ and where $\iota:\mathcal{E}^e \otimes_{B_e} \BdR\rightarrow \E_{\B+}\otimes_{\B+}\BdR$ is an isomorphism. Here, the triple corresponding to some $\E$ is given by the respective base changes of $\E$ together with the induced isomorphism.

The pullback of a vector bundle $\E_b$ via $ \Spec(\B+) \longrightarrow X $ is trivial. Indeed, the inclusion $F\hookrightarrow C\hookrightarrow \B+$ extends to an embedding of an algebraic closure $\bar F$ into $\B+$. By Lang's theorem there is a $g\in G(\bar F)$ with $gb\sigma(g^{-1})=1$, which induces the desired trivialization. It is well-defined up to the action of $G(F)$. The Beauville-Laszlo uniformization depends on the choice of such a trivialization. From now on we consider $\E_b$ together with a trivialization of $\E_{b,\B+}$, without explicitly mentioning it. If $b=1$, we choose the natural trivialization. In all cases, the trivialization of $\E_{b,\B+}$ induces a trivialization of $\E_{b,\B+}\otimes_{\B+}\BdR$ (i.e., an identification with $\BdR^n$ where $n$ is the rank) identifying $\E_{b,\B+}$ with the standard lattice $(\B+)^n$ in $\BdR^n$.

For each $ x \in \Gr_n (C) $ one can construct a modification $ \mathcal{E}_{b,x} $ of $ \mathcal{E}_b $ as follows. Using the trivialization of $\E_{b,\B+}$, we can write the triple corresponding to $\E_b$ as $(\mathcal{E}_{b | X \setminus \infty}, \E^{n,\tri}_{\B+},\iota)$ where $\E^{n,\tri}_{\B+}$ is the trivial bundle of rank $n$ on $ \Spec(\B+) $. Then $\E_{b,x}$ is given as the vector bundle corresponding to the triple $(\mathcal{E}_{b | X \setminus \infty}, \E^{n,\tri}_{\B+},\iota_x )$ where the isomorphism $\iota_x$ is given by the commutative diagram

\begin{center}
	\begin{tikzpicture}[scale = 1]
	\draw (0, 0) node { $\mathcal{E}_b^e \otimes_{B_e} \BdR $ };
	\draw (0, -2) node { $\mathcal{E}^{n, \tri}_{\B+} \otimes_{\B+} \BdR $. };
	\draw (4, 0) node { $\mathcal{E}^{n, \tri}_{\B+} \otimes_{\B+} \BdR $ };	
	\draw [->] (1.25, 0) -- (2.25, 0)node[midway, above]{$ \iota $};
	\draw [->] (0, -0.5) -- (0, -1.5)node[midway, right]{$ \iota_x $};	
	\draw [->] (0.5, -1.5) -- (3.5, -0.5)node[midway, right]{$~x $};	
	\end{tikzpicture}
\end{center}
Here, $B_e = H^0(X \setminus \infty, \mathcal{O}_X)$ and the map $x$ in the diagram is multiplication by a representative of $x$ on $\BdR^n$. The isomorphism class of the triple only depends on the lattice $x(\E^{n,\tri}_{\B+})\subset \BdR^n$ and is in particular independent of the choice of the representative. 

Write $ \Lambda_{x} := x(\mathcal{E}^{n, \tri}_{\B+}) $ and $ \mathcal{E}^{n, \tri}_{\BdR} := \mathcal{E}^{n, \tri}_{\B+} \otimes_{\B+} \BdR $. If $(k_1,\dotsc,k_n)$ with $k_1\geq \dotsm\geq k_n$ is the relative position of $\Lambda_x$ with respect to $(\B+)^n$, we call $ (k_1, \dotsc, k_n) $ the type of $x$.

\subsection{The Harder-Narasimhan formalism}\label{sec22}

We want to define a Harder-Narasimhan filtration for each pair $(b, x)$ where $b \in \GL_n(\breve F)$ and $ x \in \Gr_{n}(C)=\Gr_{\GL_n} (C) $ in such a way that it coincides with the Harder-Narasimhan filtration for the filtered isocrystal for $b$ with filtration induced by $x$ if $ x $ is of minuscule type. We will use the general Harder-Narasimhan formalism as for example in \cite{Andre}.

\begin{definition}
	\begin{enumerate}
	\item Let $ \Isoc $ be the category of $\sigma$-isocrystals over $\breve F$. Recall that each object is isomorphic to a pair $({\breve F}^n, b\sigma)$ for some $n$. From now on we write $\E_b$ for such a pair and identify it with the corresponding vector bundle on $X$. However, for two vector bundles $ \mathcal{E}_{b_1} $ and $ \mathcal{E}_{b_2} $ we set $$ \hom_{\Isoc} ( \mathcal{E}_{b_1}, \mathcal{E}_{b_2} ) = \{ \mathcal{E}(f) \mid f \in \hom ( (\breve F^{n_1}, b_1 \sigma), (\breve F^{n_2}, b_2  \sigma )) \},$$ the set of morphisms of isocrystals. In particular $ \Isoc $ is abelian and every exact sequence splits.
	\item Let $\Isoc^{\bullet} $ be the category whose objects are pairs ${}^{\bullet} \E$ consisting of an element $ \mathcal{E}$ in $\Isoc $ together with a filtration $ \Fil^{\bullet}\mathcal{E} $ in $\Isoc $ and whose morphisms are morphisms in $ \Isoc $ that are strictly compatible with the filtrations.
	
\item Let $\mathcal{E}_{b_1}$, resp.~$\mathcal{E}_{b_2}$ be vector bundles of rank $n$ resp.~$m$ and let $ f : \mathcal{E}_{b_1} \longrightarrow \mathcal{E}_{b_2} $ be a morphism in $ \Isoc $. Let $ x_1 \in \Gr_n (C) $ and  $x_2 \in \Gr_m (C) $. Then we have the associated $ \B+ $-lattices $ \Lambda_{x_1} $ resp.~$ \Lambda_{x_2} $ of  $\mathcal{E}_{b_1}$, resp.~$\mathcal{E}_{b_2}$. The morphism $f$ is effective with respect to $x_1$ and $x_2$ if $ f( {\Lambda}_{x_1} ) \subset {\Lambda}_{x_2} $. 

\item Using the same notation let $\tilde{\Lambda}_{x_2}$ be the intersection of $\Lambda_{x_2}$ with the isocrystal corresponding to $f(\mathcal{E}_{b_1})$. Then $f$ is strict effective if $ f( {\Lambda}_{x_1} ) = \tilde{\Lambda}_{x_2} $. 

\item We define a category $ \mathcal{C} $ whose objects are pairs $(\mathcal{E}, x)$ where $\mathcal{E}$ is a vector bundle on $X$ of some rank $n$ and where $ x \in  \Gr_n (C)$. A morphism from $(\mathcal{E}_{b_1}, x_1)$ to $(\mathcal{E}_{b_2}, x_2)$ is a morphism $ g \in \hom_{\Isoc} ( \mathcal{E}_{b_1}, \mathcal{E}_{b_2} ) $ that is effective with respect to $ x_1 $ and $x_2$. 	
\end{enumerate}
\end{definition}

\begin{remark}
The class of strict effective epimorphisms in $ \mathcal{C} $ is stable under pull-backs in $\mathcal C$ and that of strict effective monomorphisms is stable under push-forward. Moreover the category $ \mathcal{C} $ is additive and has kernels and cokernels. Indeed, let $g$ be in $\hom_{\mathcal{C}} ( (\mathcal{E}_{b_1}, x_1), (\mathcal{E}_{b_2}, x_2))$. Since the category $\Isoc$ has kernels, there exists $ (\mathcal{E}_{b'_1}, f)$, the kernel of $g$ in that category. Suppose that $ \rank (\mathcal{E}_{b'_1}) = \ell$ is not zero then $f$ induces an injective map $ \tilde{f} : \mathcal{E}_{b'_1}^e \otimes_{B_e} \BdR \longrightarrow \mathcal{E}_{b_1}^e \otimes_{B_e} \BdR $. Denote $ \Lambda_{x'_1} := \tilde{f}^{-1}(\Lambda_{x_1}) $ for some $ x'_1 \in \Gr_{\ell}(C) $. Then $ ((\mathcal{E}_{b'_1}, x'_1), f) $ is the kernel of $g$ in $\mathcal{C}$. Similarly, the category $\mathcal{C}$ also has cokernels. Therefore $ \mathcal{C} $ is an exact category whose exact sequences are exact sequences in $ \Isoc $ with strict effective morphisms.
\end{remark}

\begin{remark}\label{remexsqC}
We consider the forgetful functor
\begin{align*}
	\mathcal{F}ib : \mathcal{C} \quad &\longrightarrow \Isoc \\
	(\mathcal{E}_b, x) &\longmapsto \mathcal{E}_b.
\end{align*}
It is clearly exact and faithful.

Consider an exact sequence $$0\rightarrow (\E',x')\rightarrow (\E,x)\rightarrow (\E'',x'')\rightarrow 0$$ in $\mathcal C$. Let $\Lambda'$, $\Lambda$ and $\Lambda''$ be the associated lattices. Then $\Lambda'=\Lambda\cap \E^{'{\rm tri}}_{\BdR}$ and $\Lambda''$ is the image of $\Lambda$ in $\E^{''{\rm tri}}_{\BdR}$. In particular, we obtain an exact sequence 
\begin{equation}\label{eqexsq2}
0\rightarrow \E'_{x'}\rightarrow \E_{x}\rightarrow \E''_{x''}\rightarrow 0
\end{equation}
of vector bundles on $X$. Notice that in general the morphisms are not morphisms of the associated isocrystals. This filtration of $\E_x$ corresponds to the $\mathcal C$-filtration $\E'\hookrightarrow \E$ in the manner also explained in different terms in \cite[Lemma 2.4]{CFS}.

The above consideration also shows that $\mathcal Fib$ induces a bijection between the strict subobjects of $(\mathcal{E}_b, x)$ and the subobjects of $ \mathcal{E}_b $ in Isoc.% where a strict subobject is a subobject that can be inserted into an exact sequence.???E: isn't this the same as being a strict mono? That's why I deleted it.
\end{remark}

\begin{definition}
We consider the following rank, degree and slope functions on the category $\mathcal C$. For an object $(\E,x)$ let
\begin{align*}
	\rk ( \mathcal{E}, x ) &= \rank \mathcal{E}= \rank \E_x \\
	\deg ( \mathcal{E}, x ) &= \deg\mathcal{E} - \deg (x) =\deg \E_x\\
	\mu ( \mathcal{E}, x ) &= \dfrac{\deg ( \mathcal{E}, x )}{\rk ( \mathcal{E}, x )}
\end{align*}
where $ \deg (x) =\val_{\xi} (\det( \overline{x}))$ for any representative $\overline x$ of $x$ in $ \GL_n(\BdR) $. 
\end{definition}
\begin{lemma}
$\rk ( \mathcal{E}, x )$ and $\deg ( \mathcal{E}, x )$ are additive on exact sequences in $\mathcal C$.
\end{lemma}
\begin{proof}
The functions $\rk \E $  and $\deg \mathcal{E} $ are additive for all exact sequences in the category of vector bundles on $X$, compare \cite[Section 5.5.2.1; Thm. 6.5.2]{FF}. Applying this to the exact sequence in \eqref{eqexsq2}, we also obtain additivity of the rank and degree function above.
\end{proof}

\begin{remark}
Suppose that there is a morphism $ g : ( \mathcal{E}, x ) \longrightarrow ( \mathcal{E}', x' ) $ such that $\mathcal{F}ib ( g) $ is an isomorphism in $ \Isoc $. Then $ \deg \mathcal{E} = \deg \mathcal{E}' $. Because $g$ is effective, $ \deg(x) \geq \deg (x') $, thus $ \deg ( \mathcal{E}, x ) \leq \deg ( \mathcal{E}', x' ) $. Equality holds if and only if $g$ is strictly effective, in which case $g^{-1}$ is also an effective morphism such that $\mathcal{F}ib ( g^{-1}) $ is an isomorphism. In other words, $g$ is an isomorphism in $ \mathcal{C} $ if and only if it is strictly effective and $\mathcal{F}ib ( g) $ is an isomorphism in $ \Isoc $.
\end{remark}

The above properties of the functor $ \mathcal{F}ib $ and the functions $ \deg, \rk, \mu $ allows us to apply the Harder-Narasimhan formalism to the category $ \mathcal{C} $ (\cite[Proposition 4.2.2]{Andre}, \cite[section 5.5.1]{FF}). We thus immediately obtain the following results.

Recall that $\X=(\E,x)\in \mathcal C$ is called semi-stable if for every sub-object $\X'=(\E',x')$ we have $\mu(\X')\leq \mu(\X)$.
\begin{Proposition}
Every object $\X$ has a unique filtration in $ \mathcal{C} $
\begin{equation}\label{eqhnfilt}
0 = \X_0 \subsetneq \X_1 \subsetneq \cdots \subsetneq \X_r = \X
\end{equation}
such that
\begin{itemize}
	\item[$\bullet$] $ \X_i / \X_{i-1} $ is semi-stable for each $ 1 \leq i \leq r $,
	\item[$\bullet$] the sequence $(\mu( \X_i / \X_{i-1} ))_{1 \leq i \leq r}$ is strictly decreasing.
\end{itemize}	
We call \eqref{eqhnfilt} the Harder-Narasimhan filtration of $\X$.
\end{Proposition}

We write $ \X^{\geq \lambda } :=  \X_m $ where $ \mu( \X_m / \X_{m-1} ) \geq \lambda > \mu( \X_{m+1} / \X_m )$ and $ \X^{> \lambda } :=  \X_k $ where $ \mu( \X_k / \X_{k-1} ) > \lambda \geq \mu( \X_{k+1} / \X_k)$. Denote by $ \mathcal{C}_{\lambda} $ the full subcategory of $ \mathcal{C} $ whose objects are semi-stable objects of slope $ \lambda $ in $ \mathcal{C} $. Then by the general properties of the \hn formalism, this category is abelian and stable under extensions in $ \mathcal{C} $ (\cite[theo. 5.5.4]{FF}). 

\begin{definition}\label{polygon} 
 Let $ \X \in \mathcal{C} $ and let $ \mathcal{F}( \X ) : 0 = \X_0 \subsetneq \X_1 \subsetneq \cdots \subsetneq \X_r = \X $ be a filtration in $ \mathcal{C} $. The polygon $ P\mathcal{F}( \X ) $ attached to this filtration is the graph of the piecewise linear function defined on $[0, \rk \X]$ such that on the interval $ [\rk \X_i, \rk \X_{i+1}] $ for $i = 0, 1, \dotsc, r-1$, it is the linear function relating the points 
 \[
 ( \rk \X_i, \deg \X_i ) \quad \text{and} \quad ( \rk \X_{i+1}, \deg \X_{i+1} ).
 \]
 
 The Harder-Narasimhan polygon $ \HN(\X) $ attached to an object $ \X $ is the polygon attached to the Harder-Narasimhan filtration of $\X$.
 
 We note that the set of points below the Harder-Narasimhan polygon $ \HN(\X) $ is also the convex hull of the points with coordinates $( \rk \X_i,\deg \X_i )$. This does not hold for the polygon attached to an arbitrary filtration.
\end{definition}

We further have the following comparison theorem.
\begin{prop}[{\cite[Proposition 4.4.4]{Andre}}]\label{itm : comparison GLn}
Let $ \mathcal{F}(\X) = ( 0 = \X_0 \subsetneq \X_1 \subsetneq \cdots \subsetneq \X_r = \X ) $ be any filtration by strict subobjects in $ \mathcal{C} $. Then $ P\mathcal{F}( \X ) \leq \HN(\X)$. If equality occurs, then the filtration $\mathcal{F}( \X )$ is a refinement of the Harder-Narasimhan filtration of $\X$ and $ \X_i / \X_{i-1} $ is semi-stable for each $ 1 \leq i \leq r $.
\end{prop}
Here the partial order on the polygons is as usual defined via Mazur's inequality. In other words, $ P\mathcal{F}( \X )$ lies on or below $\HN(\X)$ with the same end points $ (0, 0) $ and $ (\rk \X, \deg \X) $.

\begin{proof}
	The first assertion is Proposition 4.4.4 in \cite{Andre}. Write the Harder-Narasimhan filtration of $ \X $ as  $0 = Y_0 \subsetneq Y_1 \subsetneq \cdots \subsetneq Y_m = \X$ and $P\mathcal{F}( \X ) = \HN(\X)$. Thus by definition of $P\mathcal{F}( \X )$, there exist the indices $ 0 < j_1 < \cdots < j_m \leq r $ such that $ \rk Y_i = \rk \X_{j_i} $ and $ \deg Y_i = \deg \X_{j_i} $ for $ 0 < i \leq m $. In particular, $ \mu (\X_{j_i} / \X_{j_{i-1}} ) = \mu ( Y_i / Y_{i-1} ) $. Since the polygon $P\mathcal{F}( \X )$ is maximal among the polygons associated with filtrations of $\X$, we see that $\X_{j_i} / \X_{j_{i-1}}$ is semi-stable. Indeed, if this is not true then there exists a strict sub-object $ Y $ such that $ \mu (Y / \X_{j_{i-1}}) > \mu(\X_{j_i} / \X_{j_{i-1}}) $. Denote by $\mathcal{F}'(\X)$ the filtration $ 0 = \X_0 \subsetneq \X_1 \subsetneq \cdots \subsetneq \X_{j_{i-1}} \subsetneq Y \subsetneq \X_{j_i} \subsetneq \cdots  \subsetneq \X_r = \X $. Then the polygon $ P\mathcal{F}'(\X) $ is strictly greater than $ P\mathcal{F}(\X) $, a contradiction. By the uniqueness of the Harder-Narasimhan filtration, we conclude that $0 = \X_0 \subsetneq \X_{j_1} \subsetneq \cdots \subsetneq \X_{j_m} = \X$ is the Harder-Narasimhan filtration of $\X$. %By the maximality of the polygon $P\mathcal{F}(\X)$, one can see that $ \X_i / \X_{i-1} $ is semi-stable for each $ 1 \leq i \leq r $.
\end{proof}

%\begin{remark}
%	\begin{enumerate}
%		\item Since $ \hom_{\Bun (X)} (\mathcal{E}_{\lambda}, \mathcal{E}_{\lambda} ) = H^0 (\mathcal{E}_0) = F $ (\cite[page 284, point 2]{FF}) then the category $ \Isoc $ is in fact equivalent to the category of isocrystals over $F$. If $x$ is of minuscule type then by the identification $ ( \B+ )^n / ( \xi \B+ )^n \cong C^n $, we see that $x$ gives rise to a filtration $ (\mathcal{F}il)_x $ of a $n$-dimensional $C$-vector space. Moreover, $ \deg (x) = \deg (\mathcal{F}il)_x $ (see \cite[Definition 1.1.7, 8.1.4]{DOR}), therefore in this case, our Harder-Narasimhan filtration and the Harder-Narasimhan filtration in the sense of \cite[Proposition 8.1.10]{DOR} coincides.
		 
%	\end{enumerate}	  	
%\end{remark}

\section{The Harder-Narasimhan formalism for modifications of $G$-bundles}\label{secHNG}

\subsection{Construction}\label{secTannaka}

In this section we use the Tannaka formalism to extend the above Harder-Narasimhan theory to $G$-bundles for any reductive group $G$ over $F$, under the assumption that the bundle $\E_b$ is trivial. In Section \ref{seccompat}, we explain how this implies the analogous theory for all semi-stable $G$-bundles $\E_b$. For all applications, the case of modifications of the trivial (or a semi-stable) $\E_b$ is by far the most relevant one. Nevertheless, it is a natural question if one can extend the theory also to modifications of general $G$-bundles $\E_b$.

As a preparation and a main step in the proof, we first show that the Harder-Narasimhan filtration for modifications of vector bundles is compatible with direct sums and tensor products. 

Denote by $ \overline{\mathcal{C}} $ the full sub-category of $ \mathcal{C} $ whose objects are pairs $ (\mathcal{E}, x) $ where $ \mathcal{E} $ is the trivial vector bundle of some rank $n$ and where $x\in \Gr_n(C)$. Note that if $(\mathcal{E}, x)$ is an object in $ \overline{\mathcal{C}} $, then so are its sub-objects. If $( \mathcal{E}, x ), ( \mathcal{E}', x' )$ are objects in $ \overline{\mathcal{C}} $, then $ (\mathcal{E} \oplus \mathcal{E}', x \oplus x') $ and $ (\mathcal{E} \otimes \mathcal{E}', x \otimes x') $ are still objects in $\overline{\mathcal{C}}$. 

\begin{prop}
Let $( \mathcal{E}, x ), ( \mathcal{E}', x' )\in \overline{\mathcal{C}} $. Then for all $\lambda\in \mathbb Q$ we have canonical functorial isomorphisms
\begin{enumerate}
\item $(\mathcal{E} \oplus \mathcal{E}', x \oplus x')^{\geq \lambda}\cong ( \mathcal{E}, x )^{\geq \lambda}\oplus ( \mathcal{E}', x' )^{\geq \lambda}$
\item $\big( (\mathcal{E}, x ) \otimes (\mathcal{E}', x')\big)^{\geq \lambda} \cong \sum_{\lambda_1 + \lambda_2 = \lambda} (\mathcal{E}, x)^{\geq \lambda_1} \otimes (\mathcal{E}', x')^{\geq \lambda_2}.$
\end{enumerate}
\end{prop}
\begin{proof}
For every $\lambda$ the quotient $ ( \mathcal{E} \oplus \mathcal{E}', x \oplus x' )^{\geq {\lambda}} / ( \mathcal{E} \oplus \mathcal{E}', x \oplus x')^{ >{\lambda} } = (\mathcal{E}, x )^{ {\lambda} } \oplus (\mathcal{E}', x')^{{\lambda} } $ is semi-stable of slope $ {\lambda} $ (or trivial if both summands are trivial) since the category $ \mathcal{C}_{\lambda} $ is abelian and stable under extensions in $ \mathcal{C} $. This implies (1).

It remains to prove (2). By Theorem 9.1.3 in \cite{Andre}, this property is equivalent to the following claim.

{\it Claim. } If $(\mathcal{E}, x)$ resp.~$(\mathcal{E}', x')$ are semi-stable objects of slopes $ \lambda_1 $ resp.~$ \lambda_2 $ then $(\mathcal{E}, x) \otimes (\mathcal{E}', x') $ is semi-stable of slope $ \lambda_1 + \lambda_2 $.
	
To prove this claim, we reduce it to results of Cornut \cite{Cornut} and Cornut-Peche Irissarry \cite{CP} and Fargues's theorem \cite[Thm. 14.1.1]{SW17}. 
Consider the category $ \HT^{\BdR} $ of Hodge-Tate modules whose objects are pairs $(V, \Xi)$ where $V$ is a finite-dimensional $ F $-vector space and where $ \Xi \subset V \otimes_{F} \BdR $ is a $ \B+ $-lattice. A morphism $ \mathcal{F} : (V, \Xi) \longrightarrow (V', \Xi') $ is an $ F $-linear morphism $ f : V \longrightarrow V' $ whose $ \BdR $-linear extension $ f_{\BdR} : V_{\BdR} \longrightarrow V'_{\BdR} $ satisfies $ f_{\BdR} (\Xi) \subset \Xi' $.

Let $ (\mathcal{E}, x) \in \overline{\mathcal{C}} $ and let $(B^n_e, (\B+)^n, \iota)$ be the triple corresponding to $\E$ where $ \iota : B^n_e \otimes_{B_e} \BdR \overset{\sim}{\longrightarrow} (\B+)^n \otimes_{\B+} \BdR $ is the canonical isomorphism and $n$ is some natural number. Moreover, let $\Lambda_x =x(\B+)^n\subseteq (\B+)^n \otimes_{\B+} \BdR$ be the lattice associated with $x$. Furthermore, $ H^0 (\mathcal{E}) = V_{\mathcal{E}} $ is an $n$-dimensional $ F $-vector space, see \cite[Section 8.2.1.1]{FF}. Thus, we can associate a pair $(V_{\mathcal{E}}, \Lambda_x)$ with $(\mathcal{E}, x)$. 

By \cite[Section 3.2.2]{CP}, this induces an exact tensor equivalence between the tensor categories $ \HT^{\BdR} $ and $ \overline{\mathcal{C}} $. There are also degree and rank functions defined on the objects of the category $ \HT^{\BdR} $, compare \cite[Section 3.2.2]{CP}. We claim that they agree with our functions rank and degree. Indeed, $\rk(V_{\mathcal{E}}, \Lambda_x) := \dim (V_{\mathcal{E}}) = n = \rk (\mathcal{E}, x) $. Moreover, let $(k_1,\dotsc, k_n)$ be the relative position of $ \Lambda_x $ with respect to the standard lattice in $\BdR^n $. Then, $ \deg (\mathcal{E}, x) = - \sum_{i=1}^{n} k_i $ and by definition, $ \deg (V_{\mathcal{E}}, \Lambda_x) $ is the degree of the filtration $ \mathcal{F}^{\bullet} $ of $ V_C $ where for each $ \lambda \in \Z $
\[
\mathcal{F}^{\lambda} := \dfrac{(\B+)^n \cap \xi^{\lambda}\Lambda + (\xi \B+)^n}{(\xi \B+)^n} \quad \text{in} \quad V_C = \dfrac{( \B+)^n}{(\xi \B+)^n}. 
\]

Hence, $ \deg (V_{\mathcal{E}}, \Lambda_x) = -  \sum_{i=1}^{n} k_i=\deg (\mathcal{E}, x) $. Therefore, the claim follows from the analogous assertion for $ \HT^{\BdR} $, which is shown in \cite[Prop.~44]{CP}. 
\end{proof}

Let $ G $ be a reductive group and consider a pair $ (\mathcal{E}, x) $ where $ \mathcal{E} $ is the trivial $ G $-bundle on $ X $ and $ x \in \Gr_G (C) $. Here, we consider $\E$ as an exact tensor functor $P_{\mathcal{E}} : \Rep_{F} G \longrightarrow \Bun (X) $ where $ \Rep_{F} G $ is the exact tensor category of algebraic representations of $G$ on finite-dimensional $ F $-vector spaces. Recall that for a finite-dimensional $F$-vector space $V$, elements of $\Gr_{\GL_V} (C)$ are in bijection with $ \B+ $-lattices in $ V \otimes_{F} \BdR $. Thus any representation $ (V, \rho) : \Rep_{F} G \longrightarrow \GL_{V} $  gives rise to a $ \B+ $-lattice $ (\rho \otimes_{F} \BdR) (x) $ of $ V \otimes_{F} \BdR $.

\begin{definition} \label{itm : G-filtration}
	Let $G$ be a reductive group.
	\begin{enumerate}
		\item Let $K$ be a field extension of $F$. A $G$-filtration of the category $ \Vect_{K} $ is the specification of a $K$-filtration on $ V\otimes_{F} K $ of every object $ (V, \rho) \in \Rep_{F} G $ satisfying the four conditions in \cite[Def.~4.2.6]{DOR} (functoriality, compatibility with tensor products, normalization, and exactness of the associated functor to the category of graded $K$-vector spaces). 
		\item Let $ \mathcal{E} $ be a $G$-bundle. An isoc-filtration of $ \mathcal{E} $ is a tensor functor $ \prescript{\bullet}{}{\mathcal{E}} : \Rep_{F} G \longrightarrow \Isoc^{\bullet} $ such that
		\begin{itemize}
			\item[(i)] $\omega_{\text{fil}} \circ \prescript{\bullet}{}{\mathcal{E}} = P_{\mathcal{E}} $ where $ \omega_{\text{fil}} : \Isoc^{\bullet}\longrightarrow \Isoc $ forgets the filtration,
			\item[(ii)] $ \omega_{\text{iso}} \circ \prescript{\bullet}{}{\mathcal{E}} $ is a $G$-filtration of $ \Vect_{\BdR} $. Here  $ \omega_{\text{iso}} : \Isoc^{\bullet}\longrightarrow \FilVect_{\BdR} $ is the functor that maps a filtered vector bundle $ \mathcal{E}^1 \subset \cdots \subset \mathcal{E}^m $ to the corresponding filtered $ \BdR $-vector space $ \mathcal{E}^{1, \tri}_{\BdR} \subset \cdots \subset \mathcal{E}^{m, \tri}_{\BdR} $, an object of the category  $ \FilVect_{\BdR} $ of filtered $\BdR$-vector spaces.
		\end{itemize}
	\end{enumerate}
\end{definition}

\begin{remark}
If $ \E $ is the trivial $G$-bundle, then for every $ (V, \rho) \in \Rep_{F} G $, the vector bundle $ P_{\E}(V, \rho) $ is trivial of rank $ \dim_F V $. Furthermore, we denote $ \omega^F_{\text{iso}} : \Isoc^{\bullet}\longrightarrow \FilVect_{F} $ the functor that maps a filtered vector bundle $ \mathcal{E}^1 \subset \cdots \subset \mathcal{E}^m $ whose terms are trivial vector bundles to the corresponding filtered $ F $-vector space $ H^0(\E_1) \subset \cdots \subset H^0(\E_m) $. Then $ \omega^F_{\text{iso}} \circ \prescript{\bullet}{}{\mathcal{E}} $ is a $G$-filtration of $ \Vect_{F} $ and by \cite[Section 8.2.1.1]{FF}, this functor factors through $ \omega_{\text{iso}} \circ \prescript{\bullet}{}{\mathcal{E}} $ by the functor that sends a $\BdR$-vector space $V_0$ to the $F$-vector space $ V_0^{\Gal(\BdR / F)} $.
\end{remark}

\begin{thm}\label{itm : HN filtration - general reductive group}
	Let $G$ be a reductive group and let $ \mathcal{E} $ be the trivial $G$-bundle. Let $ x \in \Gr_G (C) $. Then there is a unique isoc-filtration $ \prescript{\bullet(x)}{}{\mathcal{E}} $ of $ \mathcal{E} $ such that for any $ (V , \rho) \in \Rep_{F} G $, the induced filtration $ \prescript{\bullet(x)}{}{\mathcal{E}}(V, \rho) $ on $ P_{\mathcal{E}}(V, \rho) $ is the Harder-Narasimhan filtration of the pair $ (P_{\mathcal{E}}(V, \rho), P_x(V, \rho)) $ in the category $ \mathcal{C} $.
\end{thm}
\begin{proof}
	This result is proved as in the characteristic $0$ case of the proof of \cite[Thm.~5.3.1]{DOR}. Indeed, for every object $ (V, \rho) \in \Rep_{F} G $, we have an object $ \X:= ( P_{\mathcal{E}}(V, \rho), P_x(V, \rho) ) $ in the category $\overline{ \mathcal{C}} $. Consider its Harder-Narasimhan filtration
	\[
	0 = \X_0 \subset \X_1 \subset \cdots \subset \X_m = \X
	\] 
where the sequence $ \lambda_i := \mu(\X_i / \X_{i-1}) $ is decreasing. 

We then attach to this filtration the filtered vector bundle $ 0 = \mathcal{E}_0 \subset \mathcal{E}_{\lambda_1} \subset \cdots \subset \mathcal{E}_{\lambda_m} = P_{\mathcal{E}}(V, \rho) $ where $\mathcal{E}_{\lambda_i} $ is the underlying vector bundle of $\X_i$. In this way we get a functor $ \prescript{\bullet(x)}{}{\mathcal{E}} : \Rep_{F} G \longrightarrow \Isoc^{\bullet}$. One checks that this functor satisfies the required properties in Definition \ref{itm : G-filtration}. The key point is that the above Harder-Narasimhan formalism is compatible with tensor products and direct sums.
\end{proof}

%\begin{remark}
%	When $ x $ is of minuscule type $\mu$, we can identify $ \Gr_{G, \mu} (C) $ with $ G(C) / P_{\mu} (C) $. Thus the Harder-Narasimhan filtration in our setting and that in the setting of \cite[Definition 9.2.17, Theorem 9.2.18]{DOR} coincide.
%\end{remark}

\subsection{Parabolic reductions}\label{secparred}

Let $ \mathcal{E} $ be a $G$-bundle. For a parabolic subgroup $P$ of $G$, a reduction of $\E$ to $P$ is a $P$-bundle $\E_P$ with an isomorphism $\E_P\times^P G\rightarrow \E$.

\begin{const}\label{const1}
Let $\prescript{\bullet}{}{\mathcal{E}}$ be an isoc-filtration of $ \mathcal{E} $. Then by \cite[IV.2.2.5]{Saavedra} the group $P= P (\prescript{\bullet}{}{\mathcal{E}}) = \underline{\Aut}^{\otimes} (\omega_{\text{iso}} \circ \prescript{\bullet}{}{\mathcal{E}}) $ is a parabolic subgroup of $ G $ defined over $\BdR$. If $\E$ is the trivial $G$-bundle then the functor $\omega^F_{\text{iso}}$ is a fiber functor with value in the category $\Vect_F$ and $\omega_{\text{iso}}$ factors through $\omega^F_{\text{iso}}$. Thus, again by \cite[IV.2.2.5]{Saavedra}, the group $P_F = P (\prescript{\bullet}{}{\mathcal{E}}) = \underline{\Aut}^{\otimes} (\omega^F_{\text{iso}} \circ \prescript{\bullet}{}{\mathcal{E}}) $ is a parabolic subgroup of $ G $ defined over $F$ and $ P_F \otimes_F \BdR = P $.  

Consider the exact tensor functor assigning to $ (V, \rho) \in \Rep_{F} G $ the pair $ (\mathcal{E}_0(V, \rho),$ $\Fil \mathcal{E}_0(V, \rho) ) \in \Isoc^{\bullet} $ consisting of the trivial vector bundle of rank $ \dim V $ together with a filtration of trivial vector bundles such that we have $ \omega_{\text{iso}} (\prescript{\bullet}{}{\mathcal{E}} (V, \rho) ) = \omega_{\text{iso}} ( \mathcal{E}_0(V, \rho), \Fil \mathcal{E}_0(V, \rho) ) $. In this way we get an isoc-filtration $ \prescript{\bullet}{}{\mathcal{E}}_0 $ of $ \mathcal{E}_0 $ and moreover $P (\prescript{\bullet}{}{\mathcal{E}}_0)=P$. We can also view  $ \prescript{\bullet}{}{\mathcal{E}}_0$ as the functor corresponding to the trivial $P$-bundle on $X$.

Let $ \mathcal{E}_P $ be the sheaf of local isomorphisms between $\prescript{\bullet}{}{\mathcal{E}}_0$ and $ \prescript{\bullet}{}{\mathcal{E}}$. Since we have an equivalence between the category of trivial vector bundles over $X$ and the category of $ F $-vector spaces, the sheaf $ \mathcal{E}_P $ has a structure of $ P$-torsor. The $ G $-torsor $ \mathcal{E}_P \times^P G $ is indeed the sheaf of isomorphisms between $ \mathcal{E}_0$ and $ \mathcal{E}$ and thus it corresponds to the original $ G $-bundle $ \mathcal{E} $. Therefore $ \mathcal{E}_P $ is a $ P $-reduction of the $ G $-bundle $ \mathcal{E} $.

By \cite[IV.2.4]{Saavedra}, there is a cocharacter $ \v : \mathbb{G}_{m, F} \longrightarrow  P $ that splits $ \omega_{\text{iso}} \circ \prescript{\bullet}{}{\mathcal{E}} $. Further, $P$ is the parabolic subgroup associated with $v$, and thus $v$ can be seen as an element of $ X_*(Z_{M})_{\Q}^{\Gamma} $ where $ M$ is the Levi factor of $ P$ centralizing $v$.  In case that $G$ is quasi-split and that $P$ is standard, we can also view $ \v $ as a dominant element of $X_*(T)_{\Q}^{\Gamma} = X_*(A)_{\Q} $ that is central in the standard Levi subgroup $ M$ of $P$.
% (Theo. 5.3.1, 9.2.18; Def. 9.2.17, 4.2.6 in ). 
\end{const}
\begin{definition}\label{defHNG1}
	Let $ \mathcal{E} $ be a $G$-bundle. Then for each $ x \in \Gr_G (C) $, there is a parabolic subgroup $P_x := P (\prescript{\bullet(x)}{}{\mathcal{E}}) $ associated with the filtration $ \omega_{\text{iso}} \circ \prescript{\bullet(x)}{}{\mathcal{E}}$. Moreover, the cocharacter $\v_x:=\HN(\E,x)$ that splits $ \omega_{\text{iso}} \circ \prescript{\bullet(x)}{}{\mathcal{E}} $ is called the Harder-Narasimhan vector of $ x $.	
\end{definition}

\begin{remark}
Let $\E$ be a $G$-bundle on $X$, let $P$ be a parabolic subgroup of $G$, and let $\E_P$ be a reduction of $\E$ to $P$. Consider the map
\begin{align*}
	v_P : X^*(P)& \longrightarrow \Z \\
\chi & \longmapsto \deg \chi_* \E_P.
\end{align*}
Let $M$ be the Levi quotient of $P$. Then we have a bijection $X^*(P)\cong X^*(M^{{\rm ab}})$ where $M^{{\rm ab}}$ is the maximal abelian quotient of $M$ (or of $P$). Since the center of the derived group of $P$ is finite it induces a bijection $X^*(P)_{\mathbb Q}\cong X^*(Z_M)_{\mathbb Q}$. Hence $v_P$ as above can be seen as an element $v_P$ of $ X_*(Z_M)^{\Gamma}_{\Q} $. We call $ v_P $ the slope vector of $\mathcal{E}_P $.
\end{remark}
Recall that if $\E'$ is a modification of a $G$-bundle $\E$, then the isomorphism $\E'|_{X\setminus \{\infty\}}\cong \E|_{X\setminus \{\infty\}}$ induces a bijection between reductions of $\E$ to $P$ and reductions of $\E'$ to $P$.
\begin{lemma}\label{lemcompv}
Let $x\in \Gr_G(C)$. Let $\prescript{\bullet (x)}{}{\mathcal{E}}$ be the isoc-filtration of the trivial $G$-bundle $ \mathcal{E} $ corresponding to the Harder-Narasimhan filtration of $(\E,x)$, and let $P\subseteq G$ be the associated parabolic subgroup of $G$ (defined over $F$).  Let $v_1$ be the cocharacter associated with the filtration $ \omega_{\text{iso}} \circ \prescript{\bullet(x)}{}{\mathcal{E}} $. Let $\E_{x,P}$ be the reduction of $\E_x$ to $P$ associated with the reduction of $\E$ corresponding to $\prescript{\bullet(x)}{}{\mathcal{E}}$. Then the associated cocharacter $v_P$ coincides with $v_1$.
\end{lemma}
\begin{proof}
One can compare the two slope vectors on representations of $G$, so we may assume that $G=\GL_n$. We consider the Harder-Narasimhan filtration of $(\E,x)$ where $\E$ is the trivial vector bundle of rank $n$ on $X$. It induces a filtration in Isoc $\E_0=(0)\subsetneq \E_1\subsetneq \dotsm\subsetneq \E_i$ of $\E$. Let $P$ be the stabilizer of this filtration, a parabolic subgroup of $\GL_n$ defined over $F$. Its Levi quotient $M$ is a product of $i$ factors $\GL_{n_j}$ with $n_j=\rk \E_j-\rk \E_{j-1}$.

The Harder-Narasimhan polygon of $(\E,x)$ is the convex hull of the points $(\rk \E_j,\deg (\E_j)_x)$ for $j=0,\dotsc, i$ and where $(\E_j)_x$ is the $j$th step of the filtration of $\E_x$ induced by the filtration of $\E$. To compute $v_P$ we use that $X^*(P)_{\Q}\cong X^*(Z_M)_{\Q}\cong \Z^i$ with the $j$th factor being generated by the determinant on the $j$th factor $\GL_{n_j}$ of $M$. Thus the inclusion $X_*(Z_M)^{\Gamma}_{\Q}\hookrightarrow X_*(G)^{\Gamma}_{\Q}$ identifies $v_P$ with the Harder-Narasimhan polygon described above.
\end{proof}
\begin{remark}
	\begin{enumerate}
		\item Suppose that $G = \GL_n$. Let $x\in \Gr_G(C)$, let $\prescript{\bullet}{}{\mathcal{E}}$ be an arbitrary isoc-filtration of the trivial $G$-bundle $\E$ and let $ P \subset G $ be the associated parabolic subgroup of $G$ (defined over $F$). As in the above lemma, we have a cocharacter $v_P$ associated with the $P$-reduction $ \E_{x, P} $ corresponding to the $P$-reduction of $\E$ arising from $\prescript{\bullet}{}{\mathcal{E}}$. Moreover, there is a polygon $ P(\prescript{\bullet}{}{\mathcal{E}}) $, not necessarily convex, attached to $x$ and the isoc-filtration $\prescript{\bullet}{}{\mathcal{E}}$ (Def. \ref{polygon}). Let $v_1$ be the cocharacter associated with this polygon. Since we can use the Iwasawa decomposition of $x = p \cdot k $ ($ p \in P(\BdR), k \in G(\B+) $) to calculate both $v_P$ and $v_1$, we see that they are equal.
		\item Let $G$ be a reductive group and let $x, \prescript{\bullet}{}{\mathcal{E}}, P$ be as above. Hence we have a cocharacter $v_P$. Let $ \rho : G \longrightarrow \GL(V) $ be a representation of $G$ over $F$. Then $\rho_*(\prescript{\bullet}{}{\mathcal{E}})$ is an isoc-filtration of the trivial vector bundle $ \rho_* \mathcal{E} $ and it gives rise to a cocharacter $ v_{1, \rho} $. By the explanation above for the $\GL_n$ case, we see that $ \rho \circ v_P = v_{1, \rho} $. Thus, we can understand the cocharacter $ v_P $ as the generalization of the polygon associated with an arbitrary isoc-filtration when $G$ is no longer $\GL_n$.  
	\end{enumerate}
\end{remark}
 
\begin{prop}\label{propHNparred}%G quasi-split here?
Let $x\in \Gr_G (C)$. Let $\E'=\E_x$ be the associated modification of $\E=\E_1$. Then there is a unique parabolic subgroup $P$ of $G$ such that the slope vector $v'_P$ of the reduction $\E'_P$ corresponding to $\E_{1}^P$ is maximal with respect to the partial order, and $P$-regular. It coincides with the parabolic reduction corresponding (via Construction \ref{const1}) to the Harder-Narasimhan filtration of $(\E,x)$. Further, $v'_P=\HN(\E,x)$ is the Harder-Narasimhan vector as in Definition \ref{defHNG1}. We also call $\E'_P$ the canonical reduction of $(\E_1,x)$.
\end{prop}

\begin{remark}
Here, we use the following partial order on the slope vectors (which should not be confused with the usual order on dominant representatives): Let $P,P'$ be two parabolic subgroups of $G$ and let $v_P:X^*(P)\rightarrow \Z$ and $v_{P'}:X^*(P')\rightarrow \Z$ be two slope vectors. Conjugating by a suitable element of $G$ we may assume that $P$ and $P'$ contain a joint parabolic subgroup $P_0$. We view $v_P$ and $v_{P'}$ as elements of $X_*(Z_M)_{\mathbb Q}$ resp.~$X_*(Z_{M'})_{\mathbb Q}$ where $M$ and $M'$ are the Levi quotients of $P$ and $P'$. Both sets of central cocharacters are subsets of $X_*(Z_{M_0})_{\mathbb Q}$ for the Levi quotient $M_0$ of $P_0$. Then $v_P\leq v_{P'}$ if $v_{P'}-v_P$ is a non-negative linear combination of coroots in the unipotent radical of $P_0$. Since $v_P$ and $v_{P'}$ are central in the respective Levi quotients, this last condition holds for some choice of conjugates and of $P_0$ if and only if it holds for every such choice.

For a quasi-split group $G$ with a fixed pinning $B\supseteq T$, we may conjugate $P$ and $P'$ to standard parabolics. This maps $v_P$ and $v_{P'}$ to well-defined elements of $X_*(T)^{\Gamma}_{\Q}$, and the partial order then coincides with the usual partial order on coweights of $T$ (but not with the one comparing dominant representatives). 
\end{remark}

\begin{proof}[Proof of Prop. \ref{propHNparred}]
 By Theorem \ref{itm : HN filtration - general reductive group} and Construction \ref{const1}, there exists a unique parabolic subgroup $ P $ of $G$ such that the parabolic reduction of $\E'$ corresponding to $\E_1^P$ corresponds to the Harder-Narasimhan isoc-filtration of $ (\mathcal{E}, x) $. By definition, the associated slope vector is $P$-regular. It remains to show the claimed maximality.
  
Let $ Q $ be a parabolic subgroup of $G$ and consider the parabolic reduction $\E^Q_{1,x}$ corresponding to $\E_1^Q$. We have to show that the associated slope vector $v'_Q$ is bounded by $v'_P$. Possibly replacing $Q$ by a larger parabolic subgroup $Q'$ and thus $v'_Q$ by the corresponding image in the rational cocharacters of the center  of the Levi quotient of $Q'$, we may assume that $v'_Q$ is $Q$-dominant and $Q$-regular, in other words that conjugation by $v'_Q$ on the unipotent radical of $Q$ has strictly positive weights. For the existence of such $Q$ compare the claim in the proof of \cite[Thm. 5.5]{Viehmann20}.
 
We suppose first that $G = \GL_n$. The reduction $\E_1^Q$ corresponds to an isoc-filtration $ \mathcal{F}((\E, x)) $ of $ (\E, x) $ in the category $\mathcal{C}$, which gives rise to a polygon $ P\mathcal{F}((\E, x)) $. In the same way as in the proof of Lemma \ref{lemcompv}, we can identify the slope vector $v'_Q$ with the vector corresponding to $ P\mathcal{F}((\E, x)) $. Therefore from Proposition \ref{itm : comparison GLn} we have $ v'_Q \leq v_x $ and $v_x = v'_P$.
 
 Now consider general $G$ and let $ (V, \rho) $ be any algebraic representation of $G$. By the above $\GL_n$ case, we see that $ \rho (v'_Q) \leq \rho (v_x) $. Thus we have $ v'_Q \leq v_x $.
\end{proof}

\begin{remark}
Let $P$ be a parabolic subgroup of $G$, let $N$ denote its unipotent radical and fix a Levi subgroup $M$ of $P$. From the Iwasawa decomposition we obtain that 
$$\Gr_G(C)=\coprod_{\{\lambda\}_P}N(\BdR)M(\B+)\lambda(\xi)G(\B+)/G(\B+)$$ where the union is taken over all $P(\overline F)$-conjugacy classes of cocharacters of $P_{\overline F}$, and where $\lambda$ denotes any fixed representative of its conjugacy class. Note that the decomposition into this disjoint union does not depend on the choice of $M$ or the representatives $\lambda$. For a similar construction, compare also \cite[6.6]{Shen}

For $P=G$, the above decomposition coincides with the decomposition into affine Schubert cells, except that the piece for $\lambda$ corresponds to the affine Schubert cell for $\lambda^{-1}$. If $P$ is a Borel subgroup, we obtain the decomposition into semi-infinite orbits \`a la Mirkovic-Vilonen. In the same way as in \cite[VI.3]{FarguesScholze} we obtain that this decomposition induces a decomposition of $\Gr_G$ into locally closed sub-diamonds denoted $S_{\{\lambda\}_P,P}$.

Let $P'\supseteq M'$ be a second parabolic subgroup and Levi factor. Let $\{\mu\}_M$  be an $M(\overline F)$-conjugacy class of cocharacters of $M_{\overline F}$ and let  $\{\mu'\}_{M'}$  be an $M'(\overline F)$-conjugacy class of cocharacters of $M'_{\overline F}$. Let $B$ be a Borel subgroup of $G_{\bar F}$ defined over an algebraic closure, and let $T$ be a maximal torus of $G_{\bar F}$ contained in $B$. Then we write $\mu_{\dom}\leq \mu'_{\dom}$ if the $B$-dominant representatives $\mu_{\dom},\mu'_{\dom}\in X_*(T)$ satisfy the corresponding inequality. Notice that this does not depend on the choice of $B$ or $T$.

As in \cite[Thm.~3.2]{MirkovicVilonen} one shows that $S_{\{\mu\}_P,P}\cap \Gr_{G,\mu'}\neq \emptyset$ implies that $\{\mu\}_P\cap \overline{\Gr_{G,\mu'}}\neq \emptyset$, and hence that $(-\mu)_{\dom}\leq \mu'_{\dom}$.
\end{remark}

\begin{prop}\label{thmnonempty1}
\begin{enumerate}
\item Let $\{\mu\}$ be a conjugacy class of cocharacters of $G_{\overline F}$ and let $x\in \Gr_{G,\mu}(C)$. Then $\HN(\E,x)=\nu_{b'}$ for some $[b']\in B(G,\mu)$ satisfying the following condition. Let $P$ be the parabolic subgroup corresponding to the canonical reduction of $(\E,x)$ and let $M$ be a Levi subgroup. Then $[b']$ has a reduction $b'_M$ to $M$ such that $P$ is the parabolic subgroup associated with the $M$-dominant Newton point $\nu_{b'_M}^M$. Further, there is a $\lambda\in X_*(M)$ with $(-\lambda)_{\dom}\leq \mu_{\dom}$ and $\kappa_M(b'_M)=-\lambda^{\sharp_M}\in \pi_1(M)$. 
\item 	 Let $\lambda$ be as in (1), and assume in addition that $(-\lambda)_{\dom}=\mu_{\dom}$. Then the converse of (1) also holds. That is, for every $[b']\in B(G,\mu)$ as in (1) there is an $x\in \Gr_{G,\mu}(C)$ with $\HN(\E,x)=\nu_{b'}$. If $\mu$ is minuscule, this can be applied to all $\lambda$ as in (1).
\item For all $\lambda$ as in (1) there is an $x\in \Gr_{G,\leq\mu}(C)$ with $\HN(\E,x)=\nu_{b'}$.
\end{enumerate}
\end{prop}
\begin{definition}
For $G$ and $\mu$ as above let $B(G,\mu)_{\HN}$ be the set of $[b']\in B(G)$ satisfying the condition of Proposition \ref{thmnonempty1}(1). Let $B(G)_{\HN}=\bigcup_{\mu}B(G,\mu)_{\HN}$, the set of $[b']\in B(G)$ that have a reduction $b'_M$ to a Levi subgroup $M$ of some parabolic subgroup $P$ of $G$ such that $[b'_M]_M\in B(M)$ is basic and the $M$-dominant Newton point $\nu_{b'_M}$ of $b'_M$ is $P$-dominant $P$-regular in the sense that it satisfies $\langle \alpha,\nu_{b'_M}\rangle>0$ for each root $\alpha$ in the unipotent radical of $P$. 

For $[b']\in B(G,\mu)_{\HN}$ we write $\Gr_{G,\mu}(C)^{\HN=[b']}$ for the set of $x\in \Gr_{G,\mu}(C)$ with $\HN(\E,x)=\nu_{b'}$. 

For such $x$ we also write $\HN(\E,x)=[b']$.
\end{definition}
\begin{remark}
\begin{enumerate}
\item If would be interesting to know if the converse of (1) also holds in general, without the additional assumption of (2).
\item By definition, $B(G,\mu)_{\HN}$ is a subset of $B(G,\mu)$. In particular, it is finite. For $G$ quasi-split, one can show that $B(G,\mu)=B(G,\mu)_{\HN}$. 
\item Let $[b']\in B(G)$. Then by (3) of the proposition, there is an $x\in \Gr_G(C)$ with $\HN(x)=[b']$ if and only if $[b']\in B(G)_{\HN}$.
\end{enumerate}
\end{remark}

\begin{proof}[Proof of Proposition \ref{thmnonempty1}]
Assume that $x\in \Gr_{G,\mu}(C)$ with $\HN(\E,x)=v$. We consider the canonical reduction of $(\E,x)$, which corresponds to a parabolic subgroup $P$ of $G$. Let $\{\lambda\}_P$ be the $P(\overline F)$-conjugacy class of cocharacters of $P_{\overline F}$ with $x\in S_{\{\lambda\}_P,P}(C)$. Non-emptiness of the intersection $\Gr_{G,\mu}(C)\cap S_{\{\lambda\}_P}(C)$ then implies that $(-\lambda)_{\dom}\leq\mu_{\dom}$.

Let $M$ be a Levi subgroup of $P$. By Lemma \ref{lemcompv}, $v$ is the element of $X_*(Z_M)_{\mathbb Q}^{\Gamma}$ corresponding to  $-\lambda^{\sharp_{M}}\in\pi_1(M)_{\Gamma}$. Let $[b']\in B(M)$ be basic with $\kappa_{M}(b')= -\lambda^{\sharp_{M}}$. Then $\nu_{b'}$ is central in $M$ and $\kappa_{M}(b')=v^{\sharp_M}\in\pi_1(M)_{\Gamma,\Q}$. Thus $v=\nu_{b'}$. Finally, we have $\nu_{b'}\leq (-\lambda)_{\dom}\leq \mu_{\dom}$, and $\kappa_G(b')=-\lambda^{\sharp_G}=\mu^{\sharp_G}$, thus $[b']\in B(G,\mu)$. This finishes the proof of (1).

Let $\lambda$ be as in (2). By \cite[Cor.~A.10]{RapoportAppendix} and our assumption $\kappa_M(b'_M)=-\lambda^{\sharp_M}$, the Newton stratum $\Gr_{M,-\lambda,1}^{[b'_{M}]_M^*}$ is non-empty. Let $x_{M}\in \Gr_{M,-\lambda,1}^{[b'_{M}]_M^*}(C)$. Let $x$ be its image under the natural inclusion $\Gr_{M}\rightarrow \Gr_G$. Since $-\lambda\in\{\mu\}$, we have $x\in\Gr_{G,\mu,1}^{[b']^*}(C)$. By definition of $x$, the pair $(\E_1,x)$ has a reduction to $P$ of slope vector $\nu_{b'}$. By maximality of the Harder-Narasimhan reduction and by Remark \ref{remcompHNNewt} we thus have $\nu_{b'}\leq \HN(\E_1,x)\leq \nu_{b'}$, and hence $x$ is a point in the claimed Harder-Narasimhan stratum. This proves (2), and (3) follows along the same lines.
\end{proof}

Let $G_{\ad}$ be the adjoint group of $G$. By a subscript ad we denote the image of elements or subsets of $G$ in $G_{\ad}$, and similarly for other invariants of elements.

\begin{lemma}\label{lemgad}
Let $\{\mu\}$ be a conjugacy class of cocharacters of $G_{\overline F}$, and let $x\in \Gr_{G,\mu}(C)$. Then $\HN_G(x)_{\ad}=\HN_{G_{\ad}}(x_{\ad})$, and the canonical reduction of $(\E^{G_{\ad}}_1,x_{\ad})$ (to a parabolic subgroup $P'$) is obtained from the canonical reduction of $(\E^G_1,x)$ (to a subgroup $P$) via the map $P\rightarrow P/Z(G)= P'\subseteq G_{\ad}$. 
\end{lemma}

\begin{proof}
The projection map $G\rightarrow G_{\ad}$ induces a bijection between parabolic $F$-subgroups of $G$ and of $G_{\ad}$, and also identifies all other data needed in the description of slope vectors of parabolic reductions of vector bundles, and in the criterion given in Proposition \ref{propHNparred}. Thus the proposition implies the lemma.
\end{proof}

In the remainder of this section we introduce a second, finer invariant, called the Harder-Narasimhan type of a modification. It is along the lines of the notion of \hn types in \cite{DOR}, but differs from it if $\mu$ is not minuscule. Furthermore, our definition is less technical, due to our restriction to the case of modifications of the trivial bundle.

\begin{definition}
\begin{enumerate}
\item A \hn pair is a pair $(P,\{\lambda\}_P)$ consisting of a parabolic subgroup $P$ of $G$  and a $P(\overline F)$-conjugacy class $\{\lambda\}$ of cocharacters of $P_{\overline F}$ satisfying the following property. Let $M$ be a Levi factor of $P$ containing a representative $\lambda$ of the given class. Then $P=P(\av_M(-\lambda))$ where $\av_M(-\lambda)$ is the unique rational cocharacter of $M$ that is central in $M$ and whose image in $\pi_1(M)_{\Gamma}$ agrees with $-\lambda^{\sharp_M}$. Notice that this condition holds for all $\lambda, M$ if it holds for one such pair.  
\item A \hn type is a $G(F)$-conjugacy class of \hn pairs.
\item For a conjugacy class $\{\mu\}\in X_*(G)/G$ and $[b']\in B(G,\mu)_{\HN}$ let $\Theta(\mu,[b'])$ be the set of \hn types containing a pair $(P,\{\lambda\})$ such that $[b']$ has a reduction $b'_M$ to some Levi factor $M$ of $P$ that is basic in $M$, such that $-\lambda^{\sharp_M}=\kappa_M(b')\in \pi_1(M)_{\Gamma}$ and such that $(-\lambda)_{G-\dom}\leq \mu_{G-\dom}$. Here, $\lambda^{\sharp_M}$ is the image of the projection of $\lambda$ to $M$ in $\pi_1(M)_{\Gamma}$.
\item Let $x\in \Gr_{G,\mu}(C)$. Consider the canonical parabolic reduction of $(\E_1,x)$, and let $P\subseteq G$ be the associated parabolic subgroup. Let $\lambda\in X_*(P)$ with $x\in S_{\{\lambda\}_P,P}$. Then the $G(F)$-conjugacy class of $(P,\{\lambda\}_P)$ is called the Harder-Narasimhan type of $(\E,x)$.
\end{enumerate}
\end{definition}

\begin{remark}
Parabolic subgroups of $G$ are $G(F)$-conjugate if and only if they are $G(\overline F)$-conjugate. Furthermore, the conjugacy classes of parabolic subgroups of $G$ are in bijection with subsets of the set of simple roots of $G$ relative to $F$. In particular, for any given $G,\mu$ and $[b']\in B(G,\mu)_{\HN}$, the set $\Theta(\mu,[b'])$ is finite. By definition of $B(G,\mu)_{\HN}$, it is also non-empty. 

It is easy to find examples where $\Theta(\mu,[b'])$ has more than one element. For example, let $G=\GL_5$, let $\mu=(4,3,2,1,0)\in X_*(T)$ and let $P$ be the standard parabolic subgroup obtained as stabilizer of $\langle e_1,e_2\rangle\subseteq \breve F^5$. Then $-\lambda_1=(1,4,0,2,3)$ and $-\lambda_2=(2,3,0,1,4)$ are two different cocharacters with $(P,\{\lambda_i\})\in \Theta(\mu,[b'])$ for $[b']$ of Newton polygon $(\tfrac{5}{2}^{(2)},\tfrac{5}{3}^{(3)})$.
\end{remark}

\begin{lemma}
For $x\in \Gr_{G,\mu}(C)^{\HN=[b']}$, the Harder-Narasimhan type of $(\E_1,x)$ is a well-defined element of $\Theta(\mu,[b'])$.

Conversely, the Harder-Narasimhan type $(P,\{\lambda\}_P)$ of $x$ determines $\HN(\E,x)$ as being the unique class $[b']\in B(G,\mu)$ having a representative $b'_M$ in a Levi factor $M$ of $P$ that is basic in $M$ and with $\kappa_M(b'_M)=(-\lambda)^{\sharp_M}$.
\end{lemma}
\begin{proof}
The uniqueness of the canonical reduction implies that $P$ and $\lambda$ are uniquely determined. We have 
\begin{equation}\label{eqHNtype}
(-\lambda)^{\sharp_M}=\HN(\E,x)^{\sharp_M}=\kappa_M(b')\in \pi_1(M)_{\Gamma}.
\end{equation}

Further, $\av_M(-\lambda)$ is the \HN vector of $(\E,x)$, hence the corresponding parabolic subgroup is indeed $P$. Since $x\in S_{\{\lambda\}_P}(C)\cap \Gr_{G,\mu}(C)$, we have $(-\lambda)_{\dom}\leq \mu_{\dom}$.  Altogether, we have shown that the $G(F)$-conjugacy class of $(P,\lambda)$ is an element of $\Theta(\mu,[b'])$.

Because $\HN(\E_b,x)$ is basic in $M$, it is determined by its value of $\kappa_M$. This then follows again from \eqref{eqHNtype}.
\end{proof}
In this article we are mainly concerned with the Harder-Narasimhan vector of a modification $(\E_1,x)$.

\section{The Harder-Narasimhan stratification of the \BG}

\begin{definition}
We consider the (surjective) map $$\HN:|\Gr_{G}|\rightarrow B(G)_{\HN}$$ that maps any $x\in\Gr_{G}(C)$ to $\HN(\E_1,x)$, which by Proposition \ref{thmnonempty1} is indeed an element of $B(G)_{\HN}$.  One can easily check that this is well-defined in the sense that it is independent of the choice of $C$ and only depends on the element of $|\Gr_{G}|$ underlying the point $x$.

For $[b']\in B(G)_{\HN}$ let $\Gr_{G}(C)^{\HN\leq[b']}$ be the subset of all points with image $[b'']$ for some $[b'']\leq[b']$ and likewise for $\Gr_{G,\mu}(C)^{\HN\leq[b']}$, $\Gr_{G,\leq \mu}(C)^{\HN\leq[b']}$ and/or the subsets replacing $\HN\leq[b']$ by $\HN=[b']$.
\end{definition}

Let $P$ be a parabolic subgroup of $G$ and let $M$ be a Levi subgroup. Let $\lambda\in X_*(P)$. Then $[v(\lambda)]$ is defined to be the element of $B(G)$ corresponding to the unique basic element of $B(M)$ with $\kappa_M(v(\lambda))=\lambda^{\sharp_M}$. It only depends on the $P(\overline F)$-conjugacy class $\{\lambda\}_P$ and neither on the choice of $M$ nor of a representative $\lambda$.

\begin{thm}\label{thmdescr}
Let $[b']\in B(G)_{\HN}$.
\begin{enumerate}
\item We have 
\begin{equation}\label{eqexpdesc}
\Gr_{G}(C)^{\HN\geq [b']}=\bigcup_{P\subseteq G}\bigcup_{\{\{\lambda\}_P\mid [v(\lambda)]\geq[b']\}}S_{\{\lambda\}_P,P}(C)
\end{equation}
where the first union is taken over all parabolic subgroups of $G$.
\item $\Gr_{G,\mu}(C)^{\HN\geq[b']}$, resp. $\Gr_{G,\mu}(C)^{\HN=[b']}$ are the sets of $C$-valued points of a closed resp., a locally closed subspace of $|\Gr_{G,\mu}|$. They are invariant under the action of $G(F)$ on $\Gr_{G,\mu}$. Similarly, $\Gr_{G,\leq\mu}(C)^{\HN=[b']}$ is the set of $C$-valued points of a locally closed subspace of $|\Gr_{G,\leq\mu}|$, and $\Gr_{G,\leq\mu}(C)^{\HN\leq [b']}$ is the set of $C$-valued points of a closed subspace of $|\Gr_{G,\leq \mu}|$, and even of $|\Gr_{G}|$. 
\item $\HN:|\Gr_{G}|\rightarrow B(G)_{\HN}$ is lower semi-continuous.
\end{enumerate}
\end{thm}
\begin{definition}
Let $\Gr_{G,\mu}^{\HN\geq[b']}$ denote the closed, locally spatial subdiamond of $\Gr_{G,\mu}$ with $\Gr_{G,\mu}^{\HN\geq[b']}(C)=\Gr_{G,\mu}(C)^{\HN\geq[b']}$. It is called the closed \hn stratum for $[b']$ in $\Gr_{G,\mu}$.

Likewise, we define the \hn stratum for $[b']$ to be the locally closed locally spatial subdiamond $\Gr_{G,\mu}^{\HN=[b']}\subseteq \Gr_{G,\mu}$ with $\Gr_{G,\mu}^{\HN=[b']}(C)=\Gr_{G,\mu}(C)^{\HN=[b']}$.

We also use corresponding notions for $\Gr_{G,\leq\mu}$ instead of $\Gr_{G,\mu}$.
\end{definition}

\begin{proof}[Proof of Theorem \ref{thmdescr}]
For (1) consider some $x\in \Gr_{G}(C)$ with $\HN(x)=[b'']\geq [b']$. Let $P$ be the parabolic subgroup of $G$ corresponding to the canonical reduction of $(\E,x)$, and let $\{\lambda\}_P$ be such that $x\in S_{\{\lambda\}_P,P}(C)$. Then by \cite[Lemma 3.11(1)]{Viehmann20} the slope vector of the reduction $(\E_{1,x})_P$ is equal to $[v(\lambda)]=[b'']\geq [b']$, which proves that the left hand side is contained in the right hand side of \eqref{eqexpdesc}. Conversely, let $x\in S_{\{\lambda\}_P,P}(C)$ for some $P\subset G$ and some $\{\lambda\}_P$. Then the slope vector of the parabolic reduction $(\E_{1,x})_P$ corresponding to $\E_1^P$ is (again by \cite[Lemma 3.11(1)]{Viehmann20}) equal to $[v(\lambda)]\geq [b']$. By the maximality of the Harder-Narasimhan vector, this implies that $\HN(\E,x)\geq [b']$.

(3) follows immediately from (2), so it remains to prove (2). For this, it is enough to show that the intersection of $\Gr_{G,\leq\mu}$ with the right hand side of \eqref{eqexpdesc} is closed in $\Gr_{G,\leq \mu}$. Since $S_{\{\lambda\}_P,P}$ is a group orbit, its closure is a union of other $S_{\{\lambda'\}_P,P}$ for the same $P$. For $x\in S_{\{\lambda\}_P,P}$ we have $\kappa_M(\pr_M(x))=\lambda^{\sharp_M}\geq_P \kappa_M(b')$ where $a\geq_P a'$ for $a,a'\in \pi_1(M)_{\Gamma}$ if $a-a'$ is a non-negative linear combination of coroots for the unipotent radical of $P$. The condition $\kappa_M(\pr_M(x))\geq a$ for some fixed $a$ being a closed condition, we see that also $(\lambda')^{\sharp_M}\geq_P \kappa_M(b')$ for all $S_{\{\lambda'\}_P,P}\subseteq \overline{S_{\{\lambda\}_P,P}}$. In other words, $[v(\lambda')]\geq [b']$ for these $\lambda'$. In particular, for $P$ and $\{\lambda\}$ as on the right hand side of \eqref{eqexpdesc}, $\overline{S_{\{\lambda\}_{P},P}(C)}\cap \Gr_{G,\leq\mu}(C)$ is contained in the right hand side. Hence the intersection of $\Gr_{G,\leq\mu}$ with the right hand side of \eqref{eqexpdesc} is a finite union of subspaces of the form
\begin{equation}\label{41red}
\bigcup_{g\in G(F)}\overline{S_{\{g\lambda g^{-1}\}_{gPg^{-1}},gPg^{-1}}(C)}\cap \Gr_{G,\leq\mu}(C)
\end{equation}
for some fixed parabolic subgroup $P$ of $G$ and some fixed $\{\lambda\}_P$ with $[v(\lambda)]\geq [b']$. It remains to show that \eqref{41red} is a closed subspace. Furthermore, it is enough to take the union over all $g\in (G/P)(F)$.

By the same argument as in the proof of \cite[Prop. 8.2.1]{DOR}, one proves that $(G/P)(F)$ is a compact subspace of $G/P$. Let $x$ be a point in the closure of \eqref{41red}. Let $x_n$ be a sequence of elements of \eqref{41red} converging towards $x$ and let $g_n\in (G/P)(F)$ with $x_n\in S_{\{g_n\lambda g_n^{-1}\}_{g_nPg_n^{-1}},g_nPg_n^{-1}}(C)\cap \Gr_{G,\leq\mu}(C)$. Since $(G/P)(F)$ is compact, we may assume that the sequence $g_n$ has a limit $g$. The sequence $g_n^{-1}x_n$ converges to $g^{-1}x$, and is contained in the closed subspace $\overline{S_{\{\lambda\}_{P},P}(C)}\cap \Gr_{G,\leq\mu}(C)$. Hence $x\in \overline{S_{\{g\lambda g^{-1}\}_{gPg^{-1}},gPg^{-1}}(C)}\cap \Gr_{G,\leq\mu}(C)$.
\end{proof}
\begin{remark}\label{remdesc+}
Let $[b']\in B(G)_{\HN}$. Let $P$, $M$ and $b'_M$ be as in the definition of $B(G)_{\HN}$. 
Using the same argument as at the beginning of the proof of Theorem \ref{thmdescr}, we also obtain that 
\begin{equation*}
\Gr_{G}(C)^{\HN=[b']}=\bigcup_{g\in G(F)/P(F)}\bigcup_{\{\{\lambda\}_{gPg^{-1}}\mid [v(\lambda)]=[b']\}}S_{\{\lambda\}_{gPg^{-1}},gPg^{-1}}(C)
\end{equation*}
where the first union is taken over all parabolic subgroups of $G$.
\end{remark}

By definition, for any $G$ and $\{\mu\}$, the semi-stable Harder-Narasimhan stratum in $\Gr_{G,\mu}$ coincides with the weakly admissible locus $\Gr_{G,\mu,1}^{\wa}$ in the sense of \cite[Def.~4.1]{Viehmann20}. Indeed, if $x$ is in the semi-stable Harder-Narasimhan stratum then $\HN(\E_{1}, x)(\chi) = 0 $ for any parabolic subgroup $P$ and any character $\chi \in X^*(P / Z_G) $. Moreover, if $ \HN(\E_{1}, x) - v((\E_{1,x})_P) $ is a non-negative rational linear combination of positive absolute roots then for any $P$-dominant character $ \chi \in X^*(P / Z_G) $, we have $ \HN(\E_{1}, x)(\chi) - v((\E_{1,x})_P)(\chi) \geq 0 $ and hence $v((\E_{1,x})_P)$ is non-positive.

Furthermore, the theorem yields the following description of the complement of the weakly admissible locus.
\begin{cor}
Let $b_0\in B(G,\mu)$ be the basic element. Then $$\Gr_{G,\mu}\setminus\Gr_{G,\mu}^{\wa}=\bigcup_{P\subseteq G}\bigcup_{\{\{\lambda\}_P\mid [v(\lambda)]\nleq[b_0]\}}S_{\{\lambda\}_P,P}\cap \Gr_{G,\mu}.$$
\end{cor}

\begin{remark}
In general, $\Gr_{G,\mu}^{\HN\geq [b']}$ is not equal to the closure of $\Gr_{G,\mu}^{\HN=[b']}$. Indeed, let $G=\GL_7$, let $B$ be the Borel subgroup of upper triangular matrices and let $T$ be the diagonal torus. Let $\mu=(1,1,1,1,0,0,0)\in X_*(T)_{\dom}$. Since $\mu$ is minuscule, we may use the Bialynicki-Birula isomorphism $\BB_{\mu}$ (compare \eqref{eqBB}) to identify $\Gr_{G,\mu}$ with $\Fl(G,\mu)^{\diamond}$, the diamond corresponding to the flag variety for $G$ and $\mu$. Let $[b']\in B(G,\mu)$ with Newton vector $({\tfrac{2}{3}}^{(3)},{\tfrac{1}{2}}^{(4)})$ and let $[b'']\geq [b']$ with Newton vector $(1,{\tfrac{3}{5}}^{(5)},0)$. Since $G$ is split, both elements are in $B(G)_{\HN}$. Consider the standard parabolic subgroup $P$ corresponding to $\nu_{b''}$ and its standard Levi subgroup $M$, and let $\lambda_0=(-1,0,0,-1,-1,-1,0)\in X_*(M)$. Then $\lambda_0=w.(-\mu)_{\dom}$ for some $w\in S_7$ with $\ell(w)=6$. We have $S_{\{\lambda_0\}_P,P}\cap \Gr_{G,\mu}\subseteq \Gr_{G,\mu}^{\HN\geq [b'']}$. Let $\E_1$ be the trivial vector bundle of rank 7. 

We claim that on an open and dense subset $S'$ of $S_{\{\lambda_0\}_P,P}\cap \Gr_{G,\mu}$, any sub-vector bundle of $\E_{1,x}$ corresponding to a sub-isocrystal of $\E_1$ of rank 3 has degree $\leq 1$. Indeed, $\BB_{\mu}(S_{\{\lambda_0\}_P,P}\cap \Gr_{G,\mu})$ is the diamond associated with the Bruhat cell of $\Fl(G,\mu)$ for $w$, an irreducible subscheme of dimension $\langle 2\rho,\mu_{\dom}\rangle- \ell(w)=6$. Let $({\breve F}^n,1\cdot\sigma)$ be the isocrystal corresponding to $\E_1$, and let $N'$ be the sub-isocrystal generated by the first three standard basis vectors. Then the sub-vector bundle of $\E_{1,x}$ corresponding to $N'$ has degree $\geq 2$ if and only if $x\in \overline{S_{\{\lambda_1\}_P,P}\cap \Gr_{G,\mu}}$ where $\lambda_1=(-1,-1,0,-1,-1,0,0)$. In the same way as for $\lambda_0$, one sees that $\overline{S_{\{\lambda_1\}_P,P}\cap \Gr_{G,\mu}}$ corresponds to a closed subscheme of $\Fl(G,\mu)$ of dimension 2. In particular, its intersection with the scheme associated with $\BB_{\mu}(S_{\{\lambda_0\}_P,P}\cap \Gr_{G,\mu})$ is a closed subscheme of the latter which is of smaller dimension. The condition that any sub-vector bundle of $\E_{1,x}$ corresponding to a sub-isocrystal of $\E_1$ of rank 3 has degree $\leq 1$ is satisfied if and only if $x$ is in the complement of $$\bigcup_{g\in G(F)}\overline{S_{\{\lambda_1\}_P,P}\cap \Gr_{G,\mu}}.$$ In the same way as in the proof of Theorem \ref{thmdescr} we then see that its complement in $S_{\{\lambda_0\}_P,P}\cap \Gr_{G,\mu}$ is open and dense in the latter, which proves the claim.

On the other hand, $\Gr_{G,\mu}^{\HN=[b']}\subset \bigcup_{P'\subseteq G}\bigcup_{\{\{\lambda\}_{P'}\mid [v(\lambda)]=[b']\}}S_{\{\lambda\}_{P'},{P'}}(C)\cap \Gr_{G,\mu}(C)$ where $P'$ runs through the $G(F)$-orbit of the standard parabolic corresponding to $\nu_{b'}$. From the proof of Theorem \ref{thmdescr} we see that $\Gr_{G,\mu}^{\HN=[b'']}\subseteq \overline{\Gr_{G,\mu}^{\HN=[b']}}$ can only hold if for every $x\in \Gr_{G,\mu}^{\HN=[b'']}(C)$ there is a $P'$ and $\lambda$ as above such that $x\in \overline{S_{\{\lambda\}_{P'},{P'}}\cap \Gr_{G,\mu}}$. Fix some such $x$ and $P'$. Then $P'$ corresponds to a sub-vector bundle of $\E_1$ of rank 3 which defines a filtration in $\mathcal C$. Since $[v(\lambda)]=[b']$, for every $x'\in S_{\{\lambda\}_{P'},{P'}}\cap \Gr_{G,\mu}$, the corresponding sub-vector bundle of $\E_{1,x'}$ is of degree 2. Thus the same holds for every $x$ in the closure. In particular, every $x\in S'$ is not in the closure of $\Gr_{G,\mu}^{\HN=[b']}$, but lies in $\overline {\Gr_{G,\mu}^{\HN=[b'']}}$.

One can also explicitly construct points $x\in\Gr_{G,\mu}^{\HN=[b'']}(C)$ that lie in the closure of $\Gr_{G,\mu}^{\HN=[b']}$. Thus the closure of a Harder-Narasimhan stratum is in general not a union of strata.

A similar behaviour is shown by the Harder-Narasimhan stratification of \cite{DOR}. However, we could not find an example of this for minuscule $\mu$ (and therefore applicable to our theory) in the literature.
\end{remark}

\section{Harder-Narasimhan-strata in flag varieties, and classical points}\label{seccompDOR}

Let $\{\mu\}$ be a conjugacy class of cocharacters of $G_{\overline F}$, and let $\mu$ be a representative. Let $\Fl(G,\mu)$ be the flag variety for $G$ and $\{\mu\}$. In \cite[9.6]{DOR}, Dat, Orlik and Rapoport introduce a Harder-Narasimhan stratification of $\Fl(G,\mu)$. Their semi-stable stratum coincides with the weakly admissible locus or period domain of Rapoport and Zink.

By \cite[Prop. 19.4.2]{SW17} we have the natural Bialynicki-Birula map
\begin{equation}\label{eqBB}
\BB_{\mu}:\Gr_{G,\mu}\rightarrow \Fl(G,\mu)^{\diamond}
\end{equation}
to the diamond associated with $\Fl(G,\mu)$. It is an isomorphism if $\mu$ is minuscule. On the level of $C$-points, $\BB_{\mu}$ has the following simple description. Write $x\in \Gr_{G,\mu}(C)$ as $x=g\mu(\xi^{-1})G(\B+)/G(\B+)$ with $g\in G(\B+)$. Let $\bar g\in G(C)$ be the image of $g$ under the natural projection $G(\B+)\rightarrow G(C)$. Then $\BB_{\mu}(x)$ is the image of $\bar g$ in $\mathcal F\ell(G,\mu)$.

We fix any section $C\rightarrow \B+(C)$ of the reduction modulo $\xi$, and thus consider $ G(C)$ as a subset of $G(\B+) $. Restricting to this subset, we obtain a natural bijection 
\begin{equation}\label{eqbbc}
\BB_{\mu} : G(C) \mu (\xi^{-1}) G(\B+) / G(\B+) \longrightarrow \Fl(G,\mu)(C).
\end{equation}

%Let us recall the explicit description of $\BB_{\mu}$ on the level of $ C $-points when $ G = \GL_n $. In this case $ \mu $ is given by a tuple of integers $ (m_1, \dotsc, m_n) $ with $ m_1 \geq \cdots \geq m_n $. Then $ \Gr_{\GL_n,\mu} (C) $ parametrizes lattices $ \Xi \subset \BdR(C)^n $ of relative position $ (m_1, \dotsc, m_n)$ to the standard lattice $\B+(C)^n$. We can associate with such a lattice a descending filtration $ \Fil^{\bullet}_{\Xi} $ on the residue $ C $-vector space $ C^n = \B+ (C)^n / (\xi \B+ (C))^n $ with 
%\[
%\Fil^i_{\Xi} = \dfrac{\xi^i \Xi \cap \B+ (C)^n }{\xi^i \Xi \cap (\xi \B+ (C))^n}
%\]
%which gives us a point in $ \Fl(\GL_n,\mu)(C) $. 

\begin{thm}\label{thmcompDOR}
Assume that 
\begin{enumerate}
\item $x \in G(C) \mu (\xi) G(\B+) / G(\B+)$ for some section $C\rightarrow \B+(C)$ or
\item $x \in \Gr_{G,\mu}(C)$ for some minuscule $\mu$.
\end{enumerate}
Then the Harder-Narasimhan vector of $ (\mathcal{E}_1, x) $ equals the Harder-Narasimhan vector of $ (D, \BB_{\mu}(x)) $ \`a la \cite{DOR} where $D$ is the isocrystal associated with $ 1\in G(\breve F)$. Furthermore, the Harder-Narasimhan filtrations correspond to the same parabolic subgroup of $G$.
\end{thm}

\begin{proof}
	We use the Tannakian formalism. Thus, suppose first that $ G = \GL_n $ and that $ \mu $ is not necessary minuscule.
	
	Let $ (\mathcal{E}, x) $ be an object in $\overline{\mathcal{C}}$ such that $\E=\E_1$ is the trivial vector bundle of rank $n$ and $x$ has a representative in $\GL_n(C) \mu (\xi^{-1})$. It corresponds to the filtered isocrystal $ (\breve F^n, \Id_{\breve F^n} \sigma, \BB_{\mu}(x)) $ where we view $\BB_{\mu}(x)$ as a filtration on $\breve F^n$. Each sub-isocrystal $ (V, \Id_V\sigma) $ of $  (\breve F^n, \Id_{\breve F^n} \sigma) $ gives rise to a sub-filtered isocrystal $(V, \Id_V \sigma, \BB_{\mu}(x) \cap V \otimes_{\breve F} C)$ of $ (\breve F^n, \Id_{\breve F^n} \sigma, \BB_{\mu}(x)) $ and to a sub-object $ (\overline{\mathcal{E}}, \overline{x}) $ of $ (\mathcal{E}, x) $. More precisely, $ \overline{\mathcal{E}} $ is the trivial sub-vector bundle of $\E$ of rank $ \dim_{\breve F} V $ corresponding to the sub-isocrystal $(V,\Id_V\sigma)$ of $ (\breve F^n, \Id_{\breve F^n} \sigma)$. By Beauville-Laszlo's gluing theorem, $ \mathcal{E}^{\tri}_{\B+} \simeq \breve F^n \otimes_{\breve F} \B+ $ and $ \overline{\mathcal{E}}^{\tri}_{\B+} \simeq V \otimes_{\breve F} \B+ $. Then $\overline x$ is given by $ \overline{x} \cdot\overline{\mathcal{E}}^{\tri}_{\B+}  = V \otimes_{\breve F} \BdR \cap x\cdot(\B+)^n.$
	
	By definition we have that $ \rank (\overline{\mathcal{E}}, \overline{x}) = \rank(V, \Id_V \sigma, \BB_{\mu}(x) \cap V \otimes_{\breve F} C) = \dim_{\breve F} V $, that $ \deg (\overline{\mathcal{E}}, \overline{x}) = -\deg (\overline{x}) $ and that $ \deg (V, \Id_V \sigma, \BB_{\mu}(x) \cap V \otimes_{\breve F} C) = \deg (\BB_{\mu}(x) \cap V \otimes_{\breve F} C ) $. It remains to show that 
\begin{equation}\label{eqthmhnhn}
 \deg (\overline{x}) = \deg ( \BB_{\mu}(x) \cap V \otimes_{\breve F} C ) .
 \end{equation}
	
Let $P\subset G$ be the stabilizer of $V$, a parabolic subgroup of $G$. Because $V$ is stable under $\sigma$, the subgroup $P$ is defined over $F$. Its Levi quotient is $M=\GL(V)\times \GL(\breve F^n/V)$. We consider a representative of $x$ of the form $g\mu(\xi^{-1})$ with $g\in \GL_n(C)$. Using the Bruhat decomposition we can write $g=p_1wp_2$ with $p_1\in P(C)$, $p_2\in P_{\mu}(C)$ and $w\in W$, the Weyl group of $G$. Since $\mu(\xi)p_2\mu(\xi^{-1})\in G(\B+)$, we can replace the representative of $x$ by one of the form $p_1w\mu(\xi^{-1})$, or $p_1{}^w\mu(\xi^{-1})\in S_{\{{}^w\mu^{-1}\}_{M}}$. From this last description we see that $\bar x$ is of the form $m_1({}^w\mu(\xi^{-1}))_1$ where $m_1$ is the image of $p_1$ in the first factor of $M$, and $({}^w\mu(\xi^{-1}))_1$ is the corresponding image of ${}^w\mu(\xi^{-1})$. Its degree is the degree of $({}^w\mu(\xi^{-1}))_1$. On the other hand we have $\BB_{\mu}(x)=p_1wP_{\mu}(C)/P_{\mu}(C)$. Again we obtain that $\BB_{\mu}(x)\cap V\otimes_{\breve F}C$ is the filtration corresponding to $m_1({}^w\mu(\xi^{-1}))_1$, hence the two degrees agree as claimed.
	
Therefore our Harder-Narasimhan filtration of $ (\mathcal{E}, x)$ and the Harder-Narasimhan filtration of $ (\breve F^n, \Id_{\breve F^n} \sigma, \BB_{\mu}(x))$ \`a la \cite{DOR} are identified by the map $\BB_{\mu}$.
	
The full subcategory of $\overline{\mathcal{C}}$ with objects $(\E,x)$ where $x\in G(C)\mu(\xi^{-1})$ for some $\mu$ is stable under tensor products and direct sums. Indeed, let $ (\mathcal{E}, x) $ and $ (\mathcal{E}', x') $ be objects in $\overline{\mathcal{C}}$ such that $x = A \mu (\xi^{-1}) $ resp.~$x' = B \mu' (\xi^{-1}) $ where $ A \in \GL_n(C)$ resp.~$ B \in \GL_{n'}(C)$. Then $ x \otimes x' $ is represented by the Kronecker product of matrices $ A \mu (\xi^{-1}) \otimes B \mu' (\xi^{-1}) = ( A \otimes B ) ( \mu(\xi^{-1}) \otimes \mu'(\xi^{-1}) )$. Similarly, $ (\mathcal{E}, x) \oplus (\mathcal{E}', x') =(\E\oplus \E', x\oplus x')$ with $ x \oplus x' \in \GL_{n+n'}(C)\cdot (\mu (\xi^{-1}), \mu' (\xi^{-1}) ) $.
	
Granted the above properties, we consider the case where $G$ is any reductive group. Let $ \mathcal{E}_1 $ be the trivial $G$-bundle and $ x \in G(C) \mu (\xi^{-1}) $. For every algebraic representation $ \rho : G \longmapsto \GL_V \in \Rep_{F}G $, the element $ \rho (x) $ belongs to the subset $ \GL_V(C) \rho (\mu) (\xi^{-1}) $. Hence the corresponding Harder-Narasimhan filtrations associated with $ (\rho(\mathcal{E}_1), \rho(x)) $ \`a la \cite{DOR} and by our setting coincide. In other words, the Harder-Narasimhan filtrations of $ (\mathcal{E}_1, x) $ \`a la \cite{DOR} and by our setting coincide.  
\end{proof}

\begin{remark}\label{remcompshen}
This theorem crucially uses its assumptions (1) or (2) and is not true in general. The key point is that under this assumption we can directly compare the Iwasawa decomposition and the Bruhat decomposition of the given element $g\mu(\xi^{-1})$ to prove \eqref{eqthmhnhn}. For non-minuscule $\mu$, even the weakly admissible locus in $\Gr_{G,\mu}$ (i.e., the semi-stable or basic Harder-Narasimhan stratum) does not coincide with the inverse image under $\BB_{\mu}$ of the weakly admissible locus in $\Fl(G,\mu)$ as in \cite{DOR}. An explicit example for this is given in \cite[Ex.~4.10]{Viehmann20}.
\end{remark}

We apply this to study classical points of $\Gr_{G,\mu}$, are defined analogously to the usual notion of classical points on flag varieties.

\begin{definition}
A classical point of $\Gr_G$ is a $K$-valued point for some finite extension $K$ of $\breve F$.
\end{definition}
\begin{remark}
Classical points are a particular case of the points we consider in Theorem \ref{thmcompDOR}. Indeed, let $K$ be a finite extension of $\breve{F}$. Then the $K$-valued points of $\Gr_{G}$ are $C$-valued points (for some algebraically closed complete extension of $K$) that are invariant under $\Gal(C|K)$. Thus each such point has a representative  in $G(K)\mu(\xi)^{-1}$ for some choice of a Galois-equivariant section $K\rightarrow \B+(K)$. Notice that any such section identifies $K$ with $\B+(C)^{\Gal(C|K)}\cong K$.

Recall from \eqref{eqbbc} that for all $\mu$, the Bialynicki-Birula map induces a bijection $$G(K)\mu(\xi^{-1})G(\B+(K))/G(\B+(K))\rightarrow \Fl(G,\mu)(K).$$ 
\end{remark}
We now prove a generalization of \cite[Theorem 5.2]{Viehmann20}, where a similar comparison was shown for basic $[b']$.

\begin{prop}\label{propclpt}
Let $x$ be a classical point of $\Gr_{G,\mu}$. Then for $[b']\in B(G,\mu)$ the following are equivalent.
\begin{enumerate}
\item $x\in \Gr_{G,\mu}^{[b']^*}$
\item  $x\in \Gr_{G,\mu}^{\HN=[b']}$
\item the Harder-Narasimhan vector of $\BB_{\mu}(x)\in \Fl(G,\mu)$ in the sense of \cite{DOR} is equal to $\nu_{b'}$.
\end{enumerate}
\end{prop}
\begin{proof}
We first prove equivalence of (1) and (2). We consider the canonical reduction $(\E_{1,x})_P$ of $(\E_1,x)$ where $P$ denotes the associated parabolic subgroup of $G$. Let $[b']:=\HN(\E_1,x)\in B(G,\mu)$. Then we have to show that $\E_{1,x}\cong \E_{[b']^*}$. 

Let $M$ be a Levi subgroup of $P$, and let $b'$ be a representative of $[b']$ in $M$ such that $P$ is the parabolic subgroup associated with the Newton point of $b'$. Using the same argument as in the proof of \eqref{eqthmhnhn}, we obtain a representative of $x$ of the form $p\mu(\xi^{-1})$ with $p\in P(K)$. Let $m$ be the image of $p$ in $M(K)$. By \cite[Lemma 3.10]{Viehmann20} we have $(\E_{1,x})_P\times^P M=((\E_{1}^P)\times^P M)_{m}=\E^M_{1,m}$. Isoc-filtrations of $(\E_{1,x})_P\times^P M$ (to parabolic subgroups of $M$) induce isoc-filtrations of $\E_{1,x}$ to corresponding parabolic subgroups of $G$ contained in $P$. From the maximality of $\HN(\E_1,x)$ we obtain that $(\E_1^M,m)$ is semi-stable (i.e., weakly admissible). By \cite[Thm.~5.2]{Viehmann20}, it is also admissible. In other words, $(\E_{1,x})_P\times^P M\cong \E_{[b']^*}^M$. The slope vector (in Fargues-Fontaine's sense) of $(\E_{1,x})_P\times^P M$ equals $\HN(\E,x)$ and thus is dominant with respect to $P$. By \cite[Cor.~2.9, Thm.~2.7]{MC} this implies that $(\E_{1,x})_P\times^P M$ is a reduction of $\E_{1,x}$ to $M$, that is $\E_{1,x}\cong \E_{[b']^*}^G$.

The equivalence of (2) and (3) follows from Theorem \ref{thmcompDOR}.
\end{proof}

As an application we can determine the non-emptiness pattern for \hn strata in the sense of \cite{DOR}, for the case that $b=1$.
\begin{prop}\label{propnonemptyDOR}
 Let $\{\mu\}$ be a conjugacy class of \textit{not necessary minuscule} cocharacters of $G_{\overline F}$. Then the \hn stratum for some Harder-Narasimhan vector $\nu$ in $\Fl(G,\mu,1)$  in the sense of \cite{DOR} is non-empty if and only if $\nu=\nu_{b'}$ for some $[b']\in B(G,\mu)$ satisfying the following condition. There is a parabolic subgroup $P$ of $G$ with a Levi subgroup $M$ such that $[b']$ has a representative $b'_M$ in $M$ that is basic in $M$, and such that $P$ is the parabolic subgroup associated with $\nu_{b'_M}$. Furthermore, there is a $\lambda\in X_*(P)$ in the conjugacy class of $-\mu$ and such that $\kappa_M(b'_M)=\lambda^{\sharp_M}\in \pi_1(M)_{\Gamma}$.
\end{prop}
\begin{remark}
\begin{enumerate}
\item In \cite{Orlik06}, Orlik gave an analogous non-emptiness criterion for \hn strata in the sense of \cite{DOR} for $G=\GL_n$, and arbitrary $b$.
\item  For minuscule $\mu$, Theorem \ref{thmcompDOR} implies that $B(G,\mu)_{\HN}$ consists of those $[b']$ that satisfy the condition of Proposition \ref{propclpt}. We expect that for non-minuscule $\mu$, and already for $G=\GL_n$, the set of non-empty Harder-Narasimhan strata in $\Gr_{G,\mu}$ in our sense strictly contains the one for $\Fl(G,\mu)$ in the sense of \cite{DOR}.
\end{enumerate}
\end{remark}

\begin{proof}%[Proof of Proposition \ref{thmnonemptyDOR}]
Let $C = \widehat{\overline{F}}$ and consider the embedding $C\rightarrow \B+$ induced by the inclusion $ \overline{k} \hookrightarrow C^{\flat,\circ} $ where $\overline{k}$ is the residue field of $C$. Then we have a bijection $$ G(C)\mu(\xi^{-1})G(\B+)/G(\B+) \simeq \Fl(G, \mu) (C) $$ and from the proof of \eqref{eqthmhnhn} we see that for every parabolic subgroup $P$ of $G$ we have $$ G(C)\mu(\xi^{-1})G(\B+)/G(\B+) \subset \displaystyle \coprod_{ \{\lambda\}_P:\lambda\in \{-\mu\}_G} S_{\{\lambda\}_P,P}.$$ Thus the necessity of the claimed condition follows from Proposition \ref{thmnonempty1} and Theorem \ref{thmcompDOR}.

By \cite[9.5.10]{DOR}, the semi-stable stratum in $\Fl(M,-\lambda,1)$ is non-empty and open, and also has classical points. Let $x\in M(K)\lambda(\xi)M(\B+)/M(\B+)$ be the image of such a point, for some finite extension $K$ of $\breve F$. Then by Proposition \ref{propclpt}, $x$ is in the basic Newton stratum, which by  $\kappa_M(b'_M)=\lambda^{\sharp_M}$ is the Newton stratum for $[b'_M]^{*_M}_M$. Its image in $\Gr_G$ is a classical point $x_G$ of $\Gr_G$. By definition of $x_G$, the modified bundle $\E_{1,x_G}$ has a reduction $\E_{1,x_G,P}$ to $P$ such that the associated slope vector $\nu_{b'_M}$ is $P$-dominant and central in $M$. Further, $\E_{1,x_G,P}\times^P M=\E^M_{1,x}$ is semi-stable. Thus it is a reduction of $\E_{1,x_G}$, which implies that $x_G$ is a classical point in the Newton stratum for $[b']^*$. Again by Proposition \ref{propclpt} this shows that $\BB_{\mu}(x_G)$ is in the claimed Harder-Narasimhan stratum.
\end{proof}

\section{Newton strata and the Hodge-Newton decomposition}\label{secnewt}

\begin{lemma}\label{remcompHNNewt}
Let $\E$ be the trivial $G$-bundle on $X$ and let $x\in \Gr_{G,\mu}(C)$. Then
\begin{equation}\label{eqcomphnn}
\HN(\E,x)\leq \Newt(\E_x)^*.
\end{equation} 

In particular, for any $[b']\in B(G,-\mu)$, we have the two containments
\begin{eqnarray}
\label{eqnehn}\Gr_{G,\mu}^{[b']}&\subseteq &\bigcup_{[b'']^*\leq [b']} \Gr_{G,\mu}^{\HN=[b'']}\\
\label{eqhnne}\Gr_{G,\mu}^{\HN=[b']}&\subseteq &\bigcup_{[b'']^*\geq [b']} \Gr_{G,\mu}^{[b'']}.
\end{eqnarray}
\end{lemma}
\begin{proof}
 Recall that $\Newt(\E_{x})$ is defined via a different Harder-Narasimhan formalism, this time corresponding to all parabolic reductions of $\E_{1,x}$ (not necessarily corresponding to an isoc-filtration of $\E_1$), but for the same slope function. Thus the comparison theorem for this other Harder-Narasimhan theory implies \eqref{eqcomphnn}, and the other two assertions are an immediate consequence.
\end{proof}

However, Newton strata and Harder-Narasimhan strata are in general far from being equal. For minuscule $\mu$, a description of the set of all Newton strata in $\Gr_{G,\mu}$ containing points that are semistable in the sense of our present Harder-Narasimhan formalism is given by \cite[Thm.~1.3]{Viehmann20}. 

The following is an example of a non-basic Newton stratum in some $\Gr_{G,\mu}$ that is completely contained in the weakly admissible locus (i.e., in the basic or semistable Harder-Narasimhan stratum). This disproves an expectation expressed in \cite[9.7.2 (2)]{Far19}.

\begin{ex}
Let $G_0=\GL_5$, and $[b]\in B(G_0)$ superbasic of slope $2/5$. Let $G=G_b$ be the inner form of $G_0$ corresponding to $b$. Then $G$ does not have any proper parabolic subgroups. Thus the only non-empty Harder-Narasimhan stratum in any $\Gr_{G,\mu}$ is the weakly admissible locus, which then coincides with $\Gr_{G,\mu}$. However, there are in general many non-empty Newton strata. For example consider the minuscule cocharacter $\mu=(1,1,0,0,0)$. Then $B(G,-\mu)=\{[b]^*,[b']\}$ where $[b']$ corresponds to the element of $B(\GL_5)$ with Newton slopes $-\tfrac{1}{3}$ and $-\tfrac{1}{2}$ with multiplicities 3 and 2, respectively. Thus $\Gr_{G,\mu,1}^{[b']}$ is a non-basic and non-empty Newton stratum contained in the weakly admissible locus.

More generally, let $G$ be a reductive group over $F$, let $\{\mu\}$ be a conjugacy class of cocharacters, and let $[b']\in B(G,-\mu)$. Assume that there is no non-basic $[b'']\in B(G,\mu)_{\HN}$ with $[b'']^*\leq [b']$. Then by \eqref{eqcomphnn}, we obtain that $\Gr_{G,\mu,1}^{[b']}$ is contained in the weakly admissible locus.  
\end{ex}

\begin{ex}

Let us give an example illustrating that the condition in the previous example is, however, not necessary for $\Gr_{G,\mu,1}^{[b']}$ to be contained in the weakly admissible locus. We use the compatibilities with inner twists explained in Section \ref{seccompat}. Let $ G_0= \GL_{14} $, and let $[b] \in B(G) $ such that $ \mathcal{E}_b = \mathcal{O}(\frac{5}{7}) \oplus \mathcal{O}(\frac{5}{7}) $ where $ \mathcal{O}(\lambda) $ is the stable vector bundle of slope $ \lambda $. Let $G=G_b$ be the inner form of $G_0$ corresponding to $b$. Then the strict parabolic subgroups of $G$ are all in one $G(F)$-orbit, and the associated Levi quotient is the corresponding inner form of $\GL_7\times \GL_7$.
	
Let	$ \mu = (1^{(4)}, 0^{(10)}) $. We consider the weakly admissible locus $ \Gr_{G, \mu,1}^{\text{wa}} $. Let $[b'_1], [b'_2] \in B(G_b)$ such that $ \mathcal{E}_{b'_1}$ and $ \mathcal{E}_{b'_2}$ are the $G_b$-bundles corresponding to the $\GL_{14}$-bundles  $\mathcal{O}(\frac{3}{2})^2 \oplus \mathcal{O}(\frac{4}{5})^2 $ and $ \mathcal{O}(\frac{8}{7}) \oplus \mathcal{O}(\frac{6}{7}) $ via the inner twist. Then $ [b'_1]>[b'_2]$ and $\Gr_{G,\mu,1}^{[b'_2]}$ is not totally contained in the weakly admissible locus (for example, one can construct a classical point in this Newton stratum, which then has Harder-Narasimhan vector $-\nu_{b'_2}$). We claim that $\Gr_{G,\mu,1}^{[b'_1]}\subset \Gr_{G,\mu,1}^{\wa}$.
	
	Suppose that there exists an $x \in \Gr_{G,\mu,1}^{[b_1']}(C) \setminus \Gr_{G,\mu,1}^{\wa}(C)$.  Using the inner twist between $G$ and $\GL_{14}$ we obtain the following. Let $P$ denote the standard parabolic subgroup of $\GL_{14}$ whose standard Levi factor $M$ is $\GL_7\times \GL_7$. Let $[b_1]=[b'_1b]\in B(\GL_{14})$ be the class corresponding to  $\mathcal{O}(\frac{3}{2})^2 \oplus \mathcal{O}(\frac{4}{5})^2 $. Then there is a reduction of $b$ to $P(\breve {\mathbb Q}_p)$ such that the corresponding reduction of $\E_{b'_1}$ is violating the semi-stability condition.
	
	We know that $ (\mathcal{E}_b)_P \times^P M \simeq \mathcal{O}(\frac{5}{7}) \times \mathcal{O}(\frac{5}{7}) $ and $ (\mathcal{E}_{b'_1})_P \times^P M \simeq \mathcal{O}(\frac{5}{7})_{x_1} \times \mathcal{O}(\frac{5}{7})_{x_2} $ for some $ x_i\in \Gr_{\GL_7,\mu_i}(C)$ where $ \mu_i = (1^{(n_i)}, 0^{(7 - n_i)}) $ with $n_1 + n_2 = 4$. Since the reduction contradicts weak admissibility, we have $ n_1 > 2 $. Since $ \mathcal{O}(\frac{5}{7})_{x_1} $ is a sub-vector bundle of $ \mathcal{E}_{b_1} $, the biggest slope of $ \mathcal{O}(\frac{5}{7})_{x_1} $ is equal or smaller than $\frac{3}{2}$. Moreover, the comparison between the Harder-Narasimhan vector and the Newton polygon implies that $ n_1 < 4 $, hence $n_1 = 3$ and $n_2 = 1$.
	
	Since $\mathcal{O}(\frac{5}{7})_{x_1}$ can not contain $ \mathcal{O}(\frac{3}{2})^2 $, by \cite[Cor.~2.9]{MC}, the biggest slope of $\mathcal{O}(\frac{5}{7})_{x_2}$ is not smaller than $\frac{3}{2}$. However this is not true since $ \mu_2 = (1^{(1)}, 0^{(6)}) $.  
\end{ex}

In the remainder of this section we consider the exceptional cases where certain Harder-Narasimhan strata and Newton strata coincide. This is closely related to the concept of Hodge-Newton decomposability, which we recall from \cite[7.1]{Viehmann20} and reformulate for our context. The signs we use here differ from those in \cite{Viehmann20} since our $[b']$ is such that $\nu_{b'}$ plays the role of a Harder-Narasimhan vector (either for our theory or the one of Fargues-Fontaine) whereas in loc.~cit., $\nu_{b'}$ is a Newton vector of a $G$-bundle. 

To define the notion of Hodge-Newton decomposability, recall that the choice of an inner twisting between $G$ and a quasi-split inner form $H$ allows to identify the Newton chamber of $G$ and $H$, where the Newton chamber is the set of Galois-invariant $G(\overline{F})$-conjugacy classes of homomorphisms $\mathbb {D}_{\overline F}\rightarrow G_{\overline F}$. Applying this to both $\nu_{b'}$ and $\mu^{\diamond}$ (where $\mu^{\diamond}$ is the Galois average of $\mu$), a triple $(G,[b'],\mu)$ with $[b']\in B(G,\mu)$ is called Hodge-Newton decomposable if there is a proper standard Levi subgroup $M'$ of $H$ that contains the centralizer of the dominant Newton point $\nu_{b'}$ of $[b']$ and such that $\nu_{b',\dom}\leq_{M'} (\mu^{\diamond})_{\dom}$, i.e.~$(\mu^{\diamond})_{\dom}-(\nu_{b'})_{\dom}$ is a non-negative rational linear combination of positive coroots of $M'$. If $(G,\mu)$ is fixed, then we also say that $[b']$ is Hodge-Newton decomposable. 

Let $M$ be a Levi subgroup of (a parabolic subgroup of) $G$. From the Iwasawa decomposition we obtain a map $$\pr_M:\Gr_G(C)\rightarrow \Gr_M(C)$$ mapping $x\in \Gr_G(C)$ to the unique element $x_M$ with $x\in U(\BdR)x_M$ where $U(\BdR)$ is the unipotent radical of $P$. If $x\in S_{\{\lambda\}_P,P}$ for some cocharacter $\lambda$ of $P$, then $x_M\in \Gr_{M,\lambda_{M}}$ where $\lambda_{M}\in X_*(M)$ is a representative of the conjugacy class $\{\lambda\}_P$.

The following theorem is a variant of the Hodge-Newton decomposition, stated in terms of Harder-Narasimhan vectors, and generalized to not necessarily quasi-split groups.
\begin{thm}\label{thmhone}
Let $(G,[b'],\mu)$ be Hodge-Newton decomposable with respect to some $M'$.
\begin{enumerate}
\item There is a parabolic subgroup $P$ of $G$ with a Levi subgroup $M$ both defined over $F$ and such that $M$ is an inner form of $M'$. Further, there are representatives $b'_M$ and $\mu_M$ of $[b']$ and $\{\mu\}$ such that conjugation by $\mu_{M}$ on the unipotent radical of $P$ has only non-negative weights and such that $[b'_M]\in B(M, \mu_{M})$. Then the $M$-dominant Newton point of $b'_M$ is also $P$-dominant and $P$-regular, i.e., conjugation by $\nu_{b_M}$ on the unipotent radical of $P$ has only positive weights.
\item Let $x\in \Gr_{G,\mu,1}^{[b']^*}(C)$. Let $\E_{1,x}^P$ be the reduction of $\E_{1,x}$ corresponding to $b'_M\in M$. Then 
\begin{enumerate}
\item $\E_{1,x_M}^M$ is a reduction of $\E_{1,x}$ to $M$, and $\Newt(\E_{1,x_M}^M)=[b'_M]^*_M$.
\item Let $\E_1^P$ be the reduction to $P$ of $\E_1$ that corresponds to $\E_{1,x}^P$. Then $\E_{1}^P\times^P M$ is a reduction of $\E_{1}$ to $M$. In particular, $\HN(\E_1,x)\leq_M [b']$.
\item Choose $P$ within its $G(F)$-conjugacy class in such a way that $\E_1^P$ as in (b) is the natural reduction to $P$ of the trivial $G$-bundle $\E_1$. Then $x\in S_{\{\mu_M^{-1}\},P}(C)$. Let $x_M=\pr_M(x)$. Then $x_M\in \Gr_{M,\mu_M,1}^{[b'_M]^*_M}$.
\item $\HN(\E_1,x)$ is the image of $\HN(\E_1^M,x_M)\in B(M,\mu_M)_{\HN}$ in $B(G,\mu)$.
\item $\pr_M^{-1}(\{x_M\})\cap \Gr_{G,\mu,1}=\{x\}.$
\end{enumerate}
\end{enumerate}
\end{thm}
\begin{proof}
Replacing $G$ by $G_{\ad}$ we may assume that $G$ is adjoint. Let $b\in G(\breve F)$ be a basic element such that the associated inner form $H=G_{b}$ is quasi-split. %Let $b'$ be a representative of $[b']$ and let $M'\subseteq G_{\overline F}$ be the base change to $\overline F$ of the Levi subgroup $M'$ in the definition of Hodge-Newton reducibility. Let $b'_M\in M$ be a representative of $[b']$. Further let $\mu_M\in M(\overline F)\cap \{\mu\}$ with $\mu_M^{\sharp_M}=\kappa_M(b'_M)$. 
Let $[\tilde b']\in B(H)$ be the image of $[b']$ under the induced isomorphism $B(G)\cong B(H)$. We choose a Borel subgroup and a maximal torus $H\supseteq B\supseteq T$. Let $\tilde M$ be the standard Levi subgroup of $H$ from the definition of Hodge-Newton decomposability. It contains the centralizer of the dominant Newton point of $[\tilde b']$, and is stable under the Frobenius $\sigma$. Let $\tilde b'\in\tilde M(\breve F)\cap[\tilde b']$ be such that its $\tilde M$-dominant Newton point $\nu_{\tilde b'}$ is dominant. Let $\tilde b\in \tilde M(\breve F)$ with $\kappa_{\tilde M}(\tilde b)=\kappa_{\tilde M}(\tilde b')-\mu_{\dom}^{\sharp_{\tilde M}}$ and such that $[\tilde b]_{\tilde M}\in B(\tilde M)$ is basic. We claim that $[\tilde b]_H$ corresponds to $[1]_G$ under the isomorphism $B(G)\cong B(H)$. Let $[b_0]_H\in B(H)$ be the image of $[1]_G$. We have $\kappa_G(b')=\mu^{\sharp_G}$ and hence $\kappa_H(\tilde b)=\kappa_H(\tilde b')-\mu^{\sharp_H}=\kappa_H(b_0).$ 
Further, $$\nu_{\tilde b}^{\sharp_{\tilde M}}=\nu_{\tilde b'}^{\sharp_{\tilde M}}-\mu_{\dom}^{\sharp_{\tilde M}}=\nu_{\tilde b'}^{\sharp_{\tilde M}}-\nu_{b',\dom}^{\sharp_{\tilde M}}=-\nu_b^{\sharp_{\tilde M}}=\nu_{b_0}^{\sharp_{\tilde M}}.$$ Since both $\nu_{\tilde b}$ and $\nu_{b_0}$ are central in $\tilde M$, they agree. Thus also $[\tilde b]_{H}=[b_0]_H$.

Let $G'=H_{b_0}$. Since $\tilde M$ and the corresponding parabolic subgroup $\tilde P=\tilde M B$ are stable under $b_0$ and under $\sigma$, the group $G'$ has a parabolic subgroup and Levi subgroup $P'=\tilde P_{b_0}\supseteq M'=\tilde M_{b_0}$ defined over $F$. Furthermore, the class in $B(G')$ corresponding to $[\tilde b']_H$ has a representative in $M'$ with analogous properties as in (1). Using that $G'=G_{bb_0}\cong G$, we also obtain (1) for $G$.

Now we prove (2). Assertion (a) follows from the corresponding properties of the canonical reduction. Let $\E_1^P$ be the reduction to $P$ of $\E_1$ that corresponds to $\E_{1,x}^P$ and let $v$ be the slope vector of $\E_1^P$. Since $\E_1$ is semi-stable, we have $v\leq 0$  where $0$ denotes the trivial slope vector. Let $\{\tilde \mu\}_M$ be such that $\pr_M(x)\in \Gr_{M,\tilde \mu}$. Since $x\in \Gr_{G,\mu}$ we have $\tilde{\mu}_{\dom}\leq\mu_{\dom}$. By \cite[Lemma 3.10]{Viehmann20} we have for the slope vector $v'$ of $\E_{1,x}$ that 
\begin{equation}\label{eqhonedecchain}
(v')^{\sharp_M}=\kappa_M(b'_M)=v^{\sharp_M}+\tilde{\mu}^{\sharp_M}\leq 0+\mu_M^{\sharp_M}=\nu_{b'}^{\sharp_M}=(v')^{\sharp_M}. 
\end{equation}
Thus equality holds at each step. This implies that $v^{\sharp_M}=0$. Hence $\E_1^P\times^P M$ is the trivial $M$-bundle, and a reduction of $\E_1^G$ to $M$. In particular, the reduction $\E_{1,x}^P$ also corresponds to a reduction of $(\E_1,x)$ to $P$, hence $\HN(\E_1,x)^{\sharp_M}\geq\nu_{b'}^{\sharp_M}$. Comparison with the Newton point of $\E_{1,x}$ also shows  $\HN(\E_1,x)^{\sharp_M}\leq\nu_{b'}^{\sharp_M}$, hence the last assertion of (b).

For quasi-split groups, (c) is shown in \cite[Prop. 7.8]{Viehmann20}. The general case is shown analogously, using the existence of $M$ and $P$ we proved above. A main ingredient of this proof and our proof of (e) below is \cite[Lemma 7.9]{Viehmann20}. Since for this lemma, the reformulation in the non-quasi-split case is not obvious, we prove its generalization to our context as Lemma \ref{lem79} below.

Let $P_x$ be the parabolic subgroup of $G$ corresponding to the canonical reduction of $(\E_1,x)$, and let $P_0\supseteq P_x$ be such that $P_0$ is conjugate to $P$. Let $M_0$ be its Levi quotient. For (d) we have to show that $\E_{1,x}^{P_0}=\E_{1,x}^P$ as reductions of $\E_{1,x}$. The reduction $\E_{1,x}^{P_0}$ has a slope vector $v_0$ whose image in $\pi_1(M_0)_{\Gamma}$ corresponds to $\kappa_M(\HN(\E_1,x))=\kappa_M(b'_M)$. Since $P_0$ and $P$ are conjugate, the proof of (1) shows that there is a $b\in P_0(\breve F)$ such that $[b]\in B(G)$ is basic and such that $G_b$ is quasi-split. Let $\tilde{\E}_{P_{0,b}}$ be the $P_{0,b}$-bundle corresponding to $\E_{1,x}^{P_0}$, a reduction of the $G_b$-bundle $\tilde \E$ corresponding to $\E_{1,x}$. The above condition on $v_0$ then translates into the condition that $\tilde{\E}_{P_{0,b}}$ satisfies the assumption of Theorem \ref{thmcompareff}. Since $M$ contains the centralizer of $\nu_{b'}$, the Theorem implies that $\tilde{\E}_{P_{0,b}}$ is a coarsening of the canonical reduction of $\tilde E$, and as such uniquely defined by its slope vector. Twisting back via $b^{-1}$, we obtain that $\E_{1,x}^{P_0}$ indeed agrees with $\E_{1,x}^P$ as reductions of $\E_{1,x}$, which proves (d).

(e) follows from Lemma \ref{lem79}.
\end{proof} 

\begin{lemma}\label{lem79}
Let $P\subseteq G$ be a parabolic subgroup of $G$ and let $M$ be a Levi factor. Let $\mu\in X_*(M)$ be $P$-dominant. Let $U$ be the unipotent radical of $P$. Let $\lambda\in X_*(M)$ with $\lambda^{\sharp_M}=\mu^{\sharp_M}\in \pi_1(M)_{\Gamma}$. Then
$$S_{\{\lambda\}_P,P}(C)\cap \Gr_{G,\mu}(C)\subseteq M(\B+)\mu(\xi)G(\B+)/G(\B+).$$
\end{lemma}
This lemma and its proof are a generalization of corresponding results for unramified groups in \cite{Kot03}.
\begin{proof}
The left hand side is empty unless $\lambda_{\dom}\leq \mu_{\dom}$, so we assume this.

{\it Claim.} We have $\lambda^{\sharp_M}=\mu^{\sharp_M}\in \pi_1(M)$, without taking coinvariants.

We compute the fundamental groups with respect to a maximal torus of $M_{\overline F}$ and a Borel subgroup $B\subseteq P_{\overline{F}}$. Let $\pi_1(G,M)$ be the kernel of the natural projection $\pi_1(M)\rightarrow \pi_1(G)$. In other words, we have a short exact sequence $$0\rightarrow\pi_1(G,M)\rightarrow \pi_1(M)\rightarrow\pi_1(G)\rightarrow 0.$$ Taking coinvariants, we get an exact sequence
$$\pi_1(G,M)_{\Gamma}\rightarrow \pi_1(M)_{\Gamma}\rightarrow\pi_1(G)_{\Gamma}\rightarrow 0.$$
The group $\pi_1(G,M)$ is a free abelian group generated by the images in $\pi_1(M)$ of the simple roots of $T$ in $U$. 
Since $P$, $M$ and $U$ are $\Gamma$-invariant, $\pi_1(G,M)_{\Gamma}$ is a free abelian group generated by the set of $\Gamma$-orbits on the above set of generators of $\pi_1(G,M)$, which are permuted under the action of $\Gamma$. Thus $\pi_1(G,M)_{\Gamma}$ is torsion-free, and maps injectively to $\pi_1(M)_{\Gamma}$. 

From $\lambda_{\dom}\leq \mu_{\dom}$ and the fact that $\mu$ is $P$-dominant, we obtain that $(\mu-\lambda)^{\sharp_M}\in \pi_1(M)$ is the image of a non-negative linear combination of positive coroots. In particular, it lies in the subgroup $\pi_1(G,M)$. We consider its image in $\pi_1(G,M)_{\Gamma}$. Any non-negative linear combination of positive coroots in $U$ is mapped to a similar linear combination in $\pi_1(G,M)_{\Gamma}$, and the image of $(\mu-\lambda)^{\sharp_M}$ vanishes if and only if $(\mu-\lambda)^{\sharp_M}$ itself is $0$ in $\pi_1(G,M)$. We assumed that $(\mu-\lambda)^{\sharp_M}$ is mapped to $0$ under $\pi_1(G,M)_{\Gamma}\hookrightarrow \pi_1(M)_{\Gamma}$. Hence $(\mu-\lambda)^{\sharp_M}=0$ in $\pi_1(M)$. This finishes the proof of the claim.

The remainder of the assertion not involving a Galois action, we may base change to $\overline F$, and apply \cite[Lemma 2.2]{Kot03} and its proof to conclude as in the case for split groups.
\end{proof}

From the theorem we obtain that the inclusion $M\hookrightarrow G$ induces natural maps $$\Gr_{M,\mu_M}^{[b']^*_M}(C)\rightarrow \Gr_{G,\mu}^{[b']^*}(C)$$ and, for $[b'']_M=\HN(\E_1^M,x_M)$, $$\Gr_{M,\mu_M}^{\HN=[b'']_M}(C)\rightarrow \Gr_{G,\mu}^{\HN=[b'']_G}(C)$$ that are bijections, and that are sections of the corresponding restrictions of $\pr_{M}$.

\begin{cor}\label{corhone1}
Let $x\in \Gr_{G,\mu}(C)$. Then $\Newt(\E_{1,x})^*\in B(G,\mu)$ is Hodge-Newton decomposable for some Levi $M'$ if and only if $\HN(\E_1,x)\in B(G,\mu)$ is Hodge-Newton decomposable for $M'$.
\end{cor}
\begin{proof}
Assume that $\HN(\E_1,x)\in B(G,\mu)$ is Hodge-Newton decomposable for $M'$. Then $\HN(\E_1,x)\leq \Newt(\E_{1,x})^*\leq [\mu(\xi)]$ where the Newton point of $\mu(\xi)$ is $\mu^{\diamond}$. In other words, $\mu^{\diamond}-\HN(\E_1,x)$ is a non-negative rational linear combination of positive coroots of $M'$. We have a decomposition of  $\mu^{\diamond}-\HN(\E_1,x)$ into a sum of $\mu^{\diamond}-\Newt(\E_{1,x})^*$ and $\Newt(\E_{1,x})^*-\HN(\E_1,x)$, both of which are non-negative rational linear combinations of positive coroots of $G$. Thus both summands are non-negative rational linear combinations of positive coroots of $M'$, which implies that $\Newt(\E_{1,x})^*$ is also Hodge-Newton decomposable for $M'$. The other direction follows from Theorem \ref{thmhone} (2)(b).
\end{proof}

The following Corollary is a generalization of \cite[Conj.~1 (2)]{Far19}.
\begin{cor}\label{corfarconj12}
Let $[b']\in B(G,\mu)$ be Hodge-Newton decomposable with respect to the Levi subgroup $M'$ that is the centralizer of $\nu_{b'}$ in the quasi-split inner form of $G$. Then   
\begin{equation}\label{eqfarconj12}
\Gr_{G,\mu}^{[b']^*}\subseteq\Gr_{G,\mu}^{\HN=[b']}.
\end{equation}
For the maximal element $[\mu(\xi)]$ of $B(G,\mu)$, we have equality in \eqref{eqfarconj12}. 
\end{cor}
\begin{proof}
 Let $x\in \Gr_{G,\mu}^{[b']^*}(C).$ From the theorem we obtain a parabolic subgroup $P$ with Levi subgroup $M$ such that $M$ is an inner form of $M'$ and $x_M=\pr_M(x)$ such that both $\HN(\E_{1},x)$ and the Newton point of $\E_{1,x}$ can be computed from the corresponding invariants for $(\E_1^M,x_M)$. Because $M'$ is the centralizer of the Newton point of $[b']$, we obtain that $\E_{1,x_M}^M$ is semi-stable, or in other words in the basic Newton stratum. Then by \eqref{eqcomphnn} together with minimality of the basic Newton point we obtain that $\HN(\E_1^M,x_M)=[b'_M]$, which finishes the proof of \eqref{eqfarconj12}.

Let now $[b']$ be maximal in $B(G,\mu)$. Then $[b']=[\mu(\xi)]$ clearly satisfies the assumption of the first assertion of the corollary. It remains to show that the reverse containment in \eqref{eqfarconj12} also holds for this $[b']$. But this follows immediately from \eqref{eqhnne} together with the maximality of $[b']$.
\end{proof}

Our next result proves \cite[Conj.~1 (1)]{Far19}.
\begin{prop}\label{propfullyhn}
The following are equivalent.
\begin{enumerate}
\item The Newton stratification and the Harder-Narasimhan stratification on $\Gr_{G,\mu,1}$ coincide.
\item The basic Newton stratum in $\Gr_{G,\mu,1}$ agrees with the weakly admissible locus.
\item $(G,\mu)$ is fully Hodge-Newton decomposable in the sense of  \cite[Def.~3.1]{GHN}.
\end{enumerate}
\end{prop}
There are several characterizations of being fully Hodge-Newton decomposable. For example, $(G,\mu)$ is fully Hodge-Newton decomposable if and only if every non-basic $[b']\in B(G,\mu)$ is Hodge-Newton decomposable.
\begin{proof}
Clearly (1) implies (2). The proof that (2) implies (3) is almost literally the same as in the minuscule case, compare \cite[6]{CFS}, the only difference being that one has to replace the flag varieties by affine Schubert cells in the $\B+$-Grassmannian. For more details, compare \cite[6.7]{Shen}.

It remains to show that (3) implies (1).

{\it Claim.} If $(G,\mu)$ is fully Hodge-Newton decomposable, then every non-basic $[b']\in B(G,\mu)$ is Hodge-Newton decomposable with respect to $M'$ where $M'$ is equal to the centralizer of $\nu_{b'}$ in the quasi-split inner form $H$ of $G$.

Since the claim is an assertion only involving the Newton points of elements of $B(G,\mu)$, we may replace $G$ by the quasi-split inner form of its adjoint group and thus assume that $G$ is quasi-split. Assume that there is some $[b']\in B(G,\mu)$ that does not satisfy the assertion of the claim. Let $M'$ be the centralizer of the dominant Newton point of $[b']$ and choose a corresponding representative $b'\in M'(\breve F)\cap [b']$. Then $\mu_{\dom}-\nu_{b'}\notin\Phi_{M'}$. Let $\alpha$ be a simple root of $G$ such that $\mu_{\dom}-\nu_{b'}\notin\Phi_{M_{\alpha}}$ where $M_{\alpha}$ is the standard Levi subgroup of the maximal standard parabolic corresponding to $\alpha$. Then $\kappa_{M_{\alpha}}(b')\neq\mu^{\sharp_{M_{\alpha}}}$. Let $b_0$ be basic in $M_{\alpha}$ and with $\kappa_{M_{\alpha}}(b_0)=\kappa_{M_{\alpha}}(b')$. Then $[b_0]\leq[b']$, hence $b_0\in B(G,\mu)$. Furthermore, $M_{\alpha}$ is the centralizer of $\nu_{b_0}$. Because $\kappa_{M_{\alpha}}(b_0)\neq\mu^{\sharp_{M_{\alpha}}}$ and $M_{\alpha}$ is maximal, it is not Hodge-Newton decomposable. Thus $(G,\mu)$ is not fully Hodge-Newton decomposable, which finishes the proof of the claim.

From Corollary \ref{corfarconj12} we obtain that $\Gr_{G,\mu,1}^{[b']^*}\subseteq \Gr_{G,\mu}^{\HN=[b']}$ for all $[b']\in B(G,\mu)$. Thus we also have equality in these containments.
\end{proof}
 
Our last result in this context classifies all $[b']$ for which the corresponding Newton stratum and Harder-Narasimhan stratum coincide. 
\begin{remark}\label{rem123}
Let $x\in \Gr_{G,\mu}(C)$. Recall from Corollary \ref{corhone1} that $(\Newt(\E_{1,x}))^*$ is Hodge-Newton decomposable for some $M'$ (with respect to $(G,\mu)$) if and only if the same holds for $\HN(\E_1,x)$. Let $M'$ be the smallest Levi subgroup of the quasi-split inner form of $G$ for which this is the case, and let $M$ and $x_M$ be as in Theorem \ref{thmhone}. Then by the theorem, $(\Newt(\E_{1,x}))^*=\HN(\E_1,x)$ if and only if $\Newt(\E_{1,x_M}^M)^*=\HN(\E_1^M,x_M)$. Thus, it is enough consider equality of Newton strata and Harder-Narasimhan strata for Hodge-Newton indecomposable $[b']$.
\end{remark}
\begin{cor}
Let $[b']\in B(G,\mu)$. Assume that $([b'],\mu)$ is Hodge-Newton indecomposable and that $\mu$ is minuscule. Then $\Gr_{G,\mu}^{[b']^{*}}=\Gr_{G,\mu}^{\HN=[b']}$ if and only if $(G,\mu)$ is fully Hodge-Newton decomposable and $[b']\in B(G,\mu)$ is the basic class.
\end{cor}
\begin{proof}
If $(G,\mu)$ is fully Hodge-Newton decomposable and $[b']\in B(G,\mu)$ is the basic class, then $\Gr_{G,\mu}^{[b']^*}=\Gr_{G,\mu}^{\HN=[b']}$ by Proposition \ref{propfullyhn}. Assume that $(G,\mu)$ is not fully Hodge-Newton decomposable. Then by \cite[Thm.~1.3]{Viehmann20}, every Hodge-Newton indecomposable Newton stratum intersects the basic Harder-Narasimhan stratum, which proves the other implication. 
\end{proof}

%\section{Geometric properties}\label{secgeom}

\section{Dimensions of strata}\label{secdim}

For this section we assume that $\mu$ is minuscule. For this case, we have the comparison to the Harder-Narasimhan strata of \cite{DOR}. Since little is known about the geometric properties of the \hn strata of loc.~cit., our results also give new insight for this more classical theory.

\begin{prop}\label{propdim}
Assume that $\mu$ is minuscule, and let $[b']\in B(G,\mu)_{\HN}$. Then 
\begin{eqnarray*}
\dim \Fl (G,\mu)^{\HN=[b']}&\leq&\max_{\lambda\in \Theta(\mu,[b'])}\langle 2\rho,\mu+\lambda\rangle.
\end{eqnarray*} 
Here, we identify $\lambda$ with its $M_{b'}$-dominant representative in $X_*(T)$ and $\mu$ is the representative in $X_*(T)_{\dom}$. Further, $\rho$ is the half-sum of the positive roots of $T$ in $H$.
\end{prop}
\begin{proof}
From Remark \ref{remdesc+} we obtain that $\dim \Fl (G,\mu)^{\HN=[b']}=\dim \Gr_{G,\mu}^{\HN=[b']}\leq \max_{g,\lambda}\dim S_{\{\lambda\}_{gPg^{-1}},gPg^{-1}}\cap \Gr_{G,\mu}$. Since $\mu$ is minuscule, $\lambda\in \{-\mu\}$. To compute the dimension of the right hand side we may base change to $\overline F$. The above claim is then obtained from the well-known formulas for the dimension of Bruhat cells in flag varieties.
\end{proof}

\begin{remark}
\begin{enumerate}
\item So far, there is no formula for the dimension of intersections $S_{\lambda}\cap \Gr_{G,\mu}$ in the $\B+$-Grassmannian. For the usual affine Grassmannian over $\mathbb C$, the analogous intersections (for $P$ equal to a chosen Borel subgroup of $G$ and $\mu$ the dominant representative of $\{\mu\}$) are of dimension $\langle 2\rho, \mu+\lambda\rangle$. In view of this formula, we expect that if one can establish a similar dimension theory for the $\B+$-Grassmannian, then also $\dim \Gr_{G,\mu}^{\HN=[b']}\leq\max_{\lambda\in \Theta(\mu,[b'])}\langle 2\rho,\mu+\lambda\rangle$.
\item It would be interesting to know it equality holds in the above dimension estimate.
\end{enumerate}
\end{remark}

In \cite[Conj.~2(2)]{Far19}, Fargues conjectured that the dimension of a Harder-Narasimhan stratum should agree with the dimension of the corresponding Newton stratum, which is by \cite[IV.1]{FarguesScholze} given by $\langle 2\rho,\mu-\nu_{b'}\rangle$. Our next result is a classification of all $[b']\in B(G,\mu)_{\HN}$ for which this is the case. In particular, we show that Fargues' conjectural dimension formula only holds in exceptional cases.

\begin{prop}\label{propdim2}
Let $\mu$ be minuscule and $[b']\in B(G,\mu)_{\HN}$. Then $$\dim \Fl(G,\mu)^{[b']^*}\geq\dim \Fl(G,\mu)^{\HN=[b']}$$ with equality if and only if $[b']$ is basic in the smallest Levi subgroup to which $(G,\mu,[b'])$ is Hodge-Newton decomposable.
\end{prop}
\begin{proof}
Let $M$ be the smallest standard Levi subgroup of $G$ such that $(G,\mu,[b'])$ is Hodge-Newton decomposable with respect to $M$. By Remark \ref{rem123}, Theorem \ref{thmhone} and the Bialynicki-Birula isomorphism, the projection $\Fl(G,\mu)\rightarrow \Fl(M,\mu)$ induces bijections $\Fl(G,\mu)^{[b']^*}(C)\rightarrow \Fl(M,\mu)^{[b'_M]_M^{*}}(C)$ and  $\Fl(G,\mu)^{\HN=[b']}(C)\rightarrow \Fl(M,\mu)^{\HN=[b'_M]_M}(C)$. In particular, the dimensions of corresponding strata coincide. Thus we may assume that $M=G$ and that $(G,\mu,[b'])$ is Hodge-Newton indecomposable.

If $[b']$ is basic, then $\Fl(G,\mu)^{[b']^*}\subseteq \Fl(G,\mu)^{\HN=[b']}\subseteq \Fl(G,\mu)$. Since the dimension of the basic Newton stratum agrees with that of $\Fl(G,\mu)$, the same holds for the basic Harder-Narasimhan stratum. 

Now assume that $[b']$ is non-basic. Let $M$ be the centralizer of its Newton point. From \cite[III.5]{FarguesScholze} we obtain that $\dim \Fl(G,\mu)^{[b']^*}=\langle 2\rho, \mu-\nu_{[b']^*,\dom}\rangle=\langle 2\rho, \mu-\nu_{b',\dom}\rangle$ where we write $\mu$ for the representative of $\{\mu\}$ in $X_*(T)_{\dom}$. By the above proposition we have $\dim \Fl(G,\mu)^{\HN=[b']}\leq\langle \rho, \mu+\lambda\rangle$ for some $\lambda\in \{-\mu\}$ with $-\lambda\leq_{M} \nu_{b'}$ and $\kappa_M(b')=-\lambda^{\sharp_M}$. Using that $\nu_{b'}$ is central in $M$ we obtain
\begin{eqnarray*}
\dim \Fl(G,\mu)^{[b']^*}-\dim \Fl(G,\mu)^{\HN=[b']}&\geq&\langle 2\rho, \mu-\nu_{b'}\rangle-\langle \rho, \mu+\lambda\rangle\\
&=&\langle \rho, \mu-\nu_{b'}\rangle-\langle \rho, \lambda+\nu_{b'}\rangle\\
&=&\langle \rho, \mu-\nu_{b'}\rangle-\langle \rho_M, \lambda+\nu_{b'}\rangle\\
&=&\langle \rho, \mu-\nu_{b'}\rangle-\langle \rho_M, -w_{0,M}(\lambda)-\nu_{b'}\rangle\\
&=&\langle \rho, \mu-\nu_{b'}\rangle-\langle \rho, -w_{0,M}(\lambda)-\nu_{b'}\rangle\\
&=&\langle \rho, \mu-w_{0,M}(-\lambda)\rangle.
\end{eqnarray*}
Since $-\lambda\in W.\mu$, this pairing is non-negative with equality if and only if $w_{0,M}(-\lambda)=\mu$. However, this implies that $\kappa_M(b')=-\lambda^{\sharp_M}=\mu^{\sharp_M},$ in contradiction to the Hodge-Newton indecomposability of $(G,\mu,[b'])$.
\end{proof}

\section{Compatibilities under inner twists}\label{seccompat}

In this section we explain how to generalize our results to the case of modifications of a $G$-bundle $\E_b$ for some basic $b\in G(\breve F)$ instead of the trivial bundle $\E_1$, and discuss the relation to inner twists of $G$.

Recall from Lemma \ref{lemgad} that there is a direct comparison between the Harder-Narasimhan stratification of the $\B+$-Grassmannian for $G$ and that for $G_{\ad}$. Replacing $G$ by its adjoint group we obtain that each inner form of $G$ is an inner twist by some basic element $b\in G(\breve F)$. 

Let $b\in G(\breve F)$ be basic. By $G_b$ we denote corresponding the inner form of $G$. We consider the isomorphism $\Bun_G\cong \Bun_{G_b}$ that maps any $G$-bundle $\E$ to the $G_b$-bundle of isomorphisms $\mathcal Isom_{G}(\E_b,\E)$. Here, $\E_b$ is the $G$-bundle corresponding to $b$. In particular, this identifies $\E_b=\E_b^{G}$ with the trivial $G_b$-bundle $\E_1^{G_b}$. For an explicit description in terms of $B(G)$, see also \cite[3.2.2]{Viehmann20}.

We consider the natural isomorphism $\varphi:\Gr_{G,\breve F}\rightarrow \Gr_{G_b,\breve F}$ induced by the identity map $G_{\breve F}\rightarrow G_{b,\breve F}$. It is compatible with Beauville-Laszlo uniformization in the sense that for $x\in \Gr_{G}(C)$ we have that $(\E_b^{G})_x$ is the $G$-bundle corresponding to $(\E_1^{G_b})_{\varphi(x)}$. In particular, $\varphi$ identifies corresponding Newton strata $\Gr_{G,\mu,b}^{[b']}$ and $\Gr_{G_b,\mu,1}^{[b'b^{-1}]}$. Here, a Newton stratum $\Gr_{G_b,\mu,1}^{[b'b^{-1}]}$ or $\Gr_{G,\mu,b}^{[b']}$ is non-empty if and only if $[b'b^{-1}]\in B(G_b,-\mu)$. This motivates the notation 
\begin{eqnarray*}
B(G,\mu,b)&=&\{[b']\in B(G)\mid [b'b^{-1}]\in B(G_b,-\mu)\}\\
&=&\{[b']\in B(G)\mid \kappa_G(b')
=\kappa_G(b)-\mu^{\sharp_G}, \nu_{b'}\leq \nu_b(\mu^{-1,\diamond})_{\dom}\}
\end{eqnarray*}
where $\mu^{-1,\diamond}$ is the Galois average of $\mu^{-1}$. A similar

Furthermore, the isomorphism $G_{\breve F}\rightarrow G_{b,\breve F}$ identifies parabolic subgroups of $G_b$ (defined over $F$) with parabolic subgroups of $G_{\breve F}$ that are stable under $b\sigma$. We want to use these identifications to extend our definition of the Harder-Narasimhan formalism also to modifications of basic (but not necessarily trivial) $G$-bundles for general $G$. 

Before that, we need to check compatibility with the Harder-Narasimhan formalism for any basic $b$ and $G=\GL_n$ that we established in Section \ref{sec22}.

\begin{lemma}
Let $G=\GL_n$ and let $b\in G(\breve F)$ be basic. Then the equivalence between vector bundles of rank $n$ and $G_b$-bundles, together with its analog for vector bundles of rank $n$ with a filtration in $\mathcal C$ resp.~$P_b$-bundles (for parabolic subgroups $P_b$ of $G_b$) identifies the canonical filtration of $(\E_b^{G},x)$ in the sense of Section \ref{sec22} with the canonical reduction of $(\E_1^{G_b},\varphi(x))$. 

In particular, the bijection $B(G_b)\rightarrow B(G)$ identifies the Harder-Narasimhan vector of $(\E_b^{G},x)$ with the Harder-Narasimhan vector of $(\E_1^{G_b},\varphi(x))$.
\end{lemma}

\begin{proof}
Let $P_b$ be a parabolic subgroup of $G_b$ and $\E_1^{P_b}$ be the corresponding $P_b$-reduction of the trivial $G_b$ bundle $ \E_1^{G_b} $. Let $ \E_{1, \varphi(x)}^{P_b} $ be the corresponding $P_b$-reduction of the modified $G_b$-bundle $\E_{1, \varphi(x)}^{G_b}$ and let $v_{b, P_b}$ be the corresponding slope vector. By Proposition \ref{propHNparred}, the canonical reduction of $(\E_1^{G_b},\varphi(x))$ is the unique reduction to a parabolic subgroup $P_b$ of $G_b$ such that $ v_{b, P_b} $ is maximal among the slope vectors for all possible parabolic subgroups of $G_b$.

On the other hand, we have a bijection between parabolic subgroups $P$ of $G_{\breve{F}}$ stable under $b\sigma$ and filtrations $ \prescript{\bullet}{}{(\E_b^{G},x)}$ of $(\E_b^{G},x)$ in the category $ \mathcal{C} $. Hence for any $P$ as above we have a corresponding slope vector $ v_{1, P} $ by Definition \ref{polygon}. By Proposition \ref{itm : comparison GLn}, the canonical filtration of $(\E_b^{G}, x)$ is the unique filtration such that $ v_{1, P} $ is maximal among all the filtrations of $(\E_b^{G},x)$ in $ \mathcal{C} $.     

There is a bijection between the set of parabolic subgroups of $G_b$ and parabolic subgroups of $G$ being stable under $b\sigma$. Thus if one denotes $P \subset G $ the corresponding parabolic subgroup of $ P_b \subset G_b $ then it is enough to compare $ v_{b, P_b} $ and $ v_{1, P} $. Let $ M $ be a Levi subgroup of $P$ and let $M_b$ be the corresponding Levi subgroup of $P_b$. Then the map $ \varphi $ is compatible with the Iwasawa decomposition. More precisely, if $ x = x_P \cdot g $ where $ x_P \in P(\BdR) $ and $ g \in G(\B+) $ then $ \varphi(x) = \varphi_{P}(x_P) \cdot g' $ where $ \varphi_{P}(x_P) \in P_b $ and $ g' \in G_b(\B+) $.  

By \cite[Lemma 3.11]{Viehmann20}, we see that $ v_{b, P_b} $ is central in $X_*(M_b)_{\Q}$ and thus determined by its image in $\pi_1(M_b)_{\Gamma}$, which equals $ - \kappa_{M_b}(\varphi(x_M)) $, where $x_M$ is the image of $x_P$ in $M(\BdR)$. Similarly, $ v_{1, P} $ is central in $X_*(M)_{\Q}$ and equal to the difference of the $G$-central element with image $ \kappa_G(b)$ in $\pi_1(G)_{\Gamma}$ and the element which is central in $M$ and with image $\kappa_{M}(\varphi_M (x_{M}))$. By the argument at the end of the proof of \cite[Prop. 5.2]{CFS}, these two vectors are the Newton points of corresponding elements of $B(G_b)$ resp. $B(G)$.
\end{proof}

\begin{remark}
Let $G$ be a reductive group over $F$ and let $b\in G(\breve F)$ be basic.
\begin{enumerate}
\item 
We define the Harder-Narasimhan vector of $(\E_b^G,x)$ for $x\in \Gr_{G,\mu}(C)$ to be the Harder-Narasimhan vector associated with $(\E_1^{G_b},\varphi(x))$.
\item Let $x\in \Gr_{G,\mu}(C)$ and consider the canonical reduction of $(\E_1^{G_b},\varphi(x))$. We obtain a parabolic subgroup $P'$ of $G_b$, and a reduction $(\E_{1,x}^{G_b})_{P'}$ of $\E_{1,x}$ induced by the reduction $\E_1^{P'}$. Then $P'$ corresponds to a parabolic subgroup $P$ of $G_{\breve F}$ stable under $b\sigma$. Recall that the adic Fargues-Fontaine curve is defined as $X=Y/\phi^{\Z}$ where for $S=\Spa(C^{\flat},C^{\flat,+})$, we have $$Y=\Spa~W_{\O_F}(C^{\flat +})\setminus \{[\pi]p=0\}.$$ Consider the pullback of $\E^{G_b}_{1,x}$ and $(\E_{1,x}^{G_b})_{P'}$ to $Y$. They correspond to the pullback of $(\E_b^G)_{\varphi(x)}$ to $Y$ and to a reduction to $P$ of this bundle. Using descent via $b\sigma$ one obtains a bundle on $X$. However, one has to extend the definition of parabolic reductions of $G$-bundles on $X$ to also include parabolic subgroups not defined over $F,$ but rather over $\breve F$ and stable under $b\sigma$.

One can now reformulate all of our results in this new context. Since this translation is in each case easy to carry out, but makes the notation more involved, we leave it to the reader.
\item Notice that these remarks also apply to elements $b\in [1]\in B(G)$. For such an element, we have $G\cong G_b$ and $\E_1\cong \E_b$. However, this identification leads to a different trivialization of the trivial bundle that we modify. Based on such a trivialization, all Harder-Narasimhan strata in some $\Gr_{G,\mu}$ get shifted by an element $g\in G(\breve F)$ with $g^{-1}b\sigma(g)=1$. 
\end{enumerate}
\end{remark}

\appendix
\section{A variant of Schieder's comparison theorem}

In the first part of this appendix we study a much more classical situation, namely the canonical (or Harder-Nara\-simhan) reduction of $G$-bundles on a curve. The main result, a strengthening of Schieder's comparison theorem, seems to be unknown also in this case, and the proofs are completely parallel. 

Let $k$ be an algebraically closed field and let $X$ be a smooth complete curve over $k$. Then by classical Harder-Narasimhan theory \cite{HN}, every vector bundle $\E$ on $X$ has a unique filtration $(0)=\E_0\subsetneq \E_1\subsetneq \dotsm\subsetneq \E_r=\E$, called the canonical filtration, such that the subquotients $\E_i/\E_{i-1}$ are semi-stable of some slopes $\lambda_i\in\Q$ with $\lambda_i>\lambda_{i+1}$ for all $i$.  The vector $v_{\E}\in \Q^n_+$ where $n=\rk \E$ and with components $\lambda_1,\dotsc, \lambda_r$ with multiplicities $\rk\E_i-\rk \E_{i-1}$ is called the Harder-Narasimhan polygon $HN(\E)$ of $\E$. We also associate with it the convex polygon which is the graph of the piecewise linear continuous function $[0,n]\rightarrow \R$ mapping $0$ to $0$ and whose slope on $[l-1,l]$ is equal to the $l$th component of $HN(\E)$.

Furthermore, there is the following comparison theorem.  
Let $\E'$ be a sub-vector bundle of $\E$. Then \begin{enumerate}
 \item The point $(\rk \E',\deg \E')$ lies under or on $HN(\E)$. 
\item\label{comp2} If it lies on $HN(\E)$, there is an $i$ with $\E_i\subseteq \E'\subseteq \E_{i+1}$.
 \end{enumerate} 

Let now $G$ be a reductive group over $k$. Generalizing the above theory there is for every $G$-bundle on $X$ a canonical or Harder-Narasimhan reduction to a $P$-bundle for some standard parabolic subgroup $P$ of $G$, compare for example \cite{BH}, \cite{B}. In \cite[Thm.~4.1]{S}, Schieder generalizes the first part of the above comparison theorem to this context. The first goal of the present work, Theorem \ref{thmcompare} below, is to generalize (2) to the group-theoretic context. Schieder's comparison theorem also contains a second assertion concerning refinements of canonical reductions. It is the group-theoretic version of the following statement, which is implied by \eqref{comp2} above: Assume that $(0)=\E'_0\subsetneq \E'_1\subsetneq \dotsm\subsetneq \E'_s=\E$ is a second filtration of $\E$. Associate with it a similar polygon $v$ whose slopes are the slopes of the subquotients of this filtration with corresponding multiplicities. If $v=HN(\E)$, then this new filtration is a refinement of the canonical filtration of $\E$. 

Using the classification of $G$-bundles on the Fargues-Fontaine curve, we derive in Section \ref{seccompffc} a variant of Theorem \ref{thmcompare} for this context.

\subsection{Comparing reductions of $G$-bundles}\label{seccompare}

 \subsubsection{Canonical reductions}
We recall some of the main notions regarding the canonical reduction for $G$-bundles on curves. We follow Schieder \cite[2]{S}, but replace some of his notation. For more details and proofs compare \cite{BH}, \cite{B}, \cite{S}.
 
Let $k$ be an algebraically closed field of any characteristic, and let $G$ be a reductive group over $k$. We fix a maximal torus $T$ and a Borel subgroup $B$ containing it. For a standard parabolic subgroup $P$ of $G$ we denote by $M$ its Levi quotient. Let $\Delta$ be the set of simple roots for our choice of $B$ and $T$, and $\Delta_M$ correspondingly for $M$. We denote by $\pi_1(G)$ the quotient of $X_*(T)$ by the coroot lattice. For $v\in X_*(T)_{\Q}$ let $v^{\sharp_M}$ denote its image in $\pi_1(M)_{\Q}$. We also write $v^{\sharp}$ instead of $v^{\sharp_G}$.

Let $X$ be a smooth and complete curve over $k$. As before, $G$-bundles on $X$ can be viewed as $G$-torsors on $X$ that are locally trivial for the \'etale topology, or equivalently as exact tensor functors from the category ${\rm Rep}_G$ of rational algebraic representations of $G$ to the category $\Bun_X$ of vector bundles on $X$. 

We denote by ${\rm Bun}_G$ the moduli stack of $G$-bundles on $X$, and similarly for other linear algebraic groups. For $P$ and $M$ as above we have natural maps of stacks $$\Bun_G\longleftarrow \Bun_P\longrightarrow \Bun_M$$ induced by the homomorphisms $P\hookrightarrow G$ and $P\twoheadrightarrow M$.

We have $\pi_0(\Bun_P)\cong \pi_1(M)$. For $\lambda\in \pi_1(M)$ we denote by $\Bun_{P,\lambda}$ the corresponding connected component of $\Bun_P$. 

If $H\subseteq G$ is a subgroup and ${\E}\in \Bun_G$, a reduction of ${\E}$ to $H$ is an element ${\E}_H\in\Bun_H$ such that ${\E}_H\times^H G\cong {\E}$. Let ${\E}\in \Bun_G$ and let ${\E}_P\in \Bun_P$ be a reduction to $P$. We associate with ${\E}_P$ the rational cocharacter $v\in X_*(T)_{\mathbb Q,\dom}$ defined via 
\begin{eqnarray*}
v=v({\E}_P): X^*(P)&\rightarrow &\Z\\
\chi&\mapsto& \deg \chi_*({\E}_P)
\end{eqnarray*} 
and such that $v$ is central in $M$. It is called the slope of ${\E}_P$.

\begin{remark}\label{lemdescslope}
A second description of $v$ (used in \cite{S}) is obtained by associating with ${\E}_P$ the image $\lambda_P=\lambda_P({\E}_P)\in \pi_1(M)$ under the composition $\Bun_P\rightarrow \Bun_M\rightarrow \pi_0(\Bun_M)\cong \pi_1(M)$. 

Then $v({\E}_P)\in X_*(T)_{\Q,\dom}$ is the unique element that is central in $M$ and such that $v({\E}_P)^{\sharp_M}=\lambda_P({\E}_P)$ in $\pi_1(M)_{\Q}$. In other words, it is the image of $\lambda_P({\E}_P)$ under 
\begin{equation}\label{eqliftel}
\phi_M:\pi_1(M)_{\Q}\rightarrow X_*(Z(M)^0)_{\Q}\hookrightarrow X_*(T)_{\mathbb Q}.
\end{equation}
Conversely, $\lambda_P({\E}_P)$ is uniquely determined by its image in $\pi_1(G)$, which coincides with $\lambda_G({\E})$, together with the image of $\lambda_P({\E}_P)$ in $\pi_1(M)_{\Q}$, which coincides with $v({\E}_P)^{\sharp_M}$. Indeed, this follows since the kernel of the projection map $\pi_1(M)\rightarrow \pi_1(G)$ is torsion free.
\end{remark}

For $\lambda,\lambda'\in X_*(T)_{\mathbb Q}$ we write $\lambda\preceq \lambda'$ if $\lambda'-\lambda$ is a non-negative linear combination of positive coroots. Notice that we do not assume $\lambda$ or $\lambda'$ to be dominant. 

\begin{definition}
\begin{enumerate}
\item A $G$-bundle ${\E}$ is called semistable if for every standard parabolic subgroup $P$ and every reduction ${\E}_P$ of ${\E}$ to $P$ we have $v({\E}_P)\preceq v({\E})$. We denote the semistable locus in $\Bun_G$ by $\Bun_G^{ss}$, and likewise for $\Bun_P^{ss}$, $\Bun_{G,\lambda}^{ss}$ etc.
\item Let $P$ be a standard parabolic subgroup of $G$ and $M$ its standard Levi factor. Then $v\in X_*(T)_{\mathbb Q,\dom}$ is called dominant $P$-regular if $v$ is central in $M$, and if $\langle v,\alpha\rangle>0$ for all simple roots $\alpha$ that are not in $M$. Similarly, $\lambda\in \pi_1(M)_{\mathbb Q}$ is called dominant $P$-regular if $\phi_M(\lambda)$ is dominant $P$-regular. 
\item A reduction ${\E}_P$ of a $G$-bundle is called canonical, if ${\E}_P$ is semistable and $v({\E}_P)$ is dominant $P$-regular.
\end{enumerate}
\end{definition}
Then Harder-Narasimhan theory for $G$-bundles on $X$ (\cite{BH}, \cite{B}) implies that for every $G$-bundle ${\E}$ on $X$ there is a unique standard parabolic subgroup $P$ of $G$ such that there is a canonical reduction ${\E}_P$ of ${\E}$. This reduction is then also uniquely defined.

\subsubsection{The comparison theorem}
Let $P_1,P_2$ be standard parabolic subgroups of $G$ with standard Levi subgroups $M_1,M_2$. Let ${\E}_{P_1}\in \Bun_{P_1, \lambda_{P_1}}^{ss}$ and ${\E}_{P_2}\in \Bun_{P_2, \lambda_{P_2}}$ be reductions of the same $G$-bundle ${\E}$ on the curve $X$. Assume that $ \lambda_{P_1}=\lambda$ is dominant $P_1$-regular, i.e.~${\E}_{P_1}$ is the canonical reduction of ${\E}$. Recall that by Remark \ref{lemdescslope} we have $v({\E}_{P_i})=\phi_{P_i}(\lambda_i)$. Then the first assertion of \cite[Thm.~4.1]{S} shows $v({\E}_{P_2})\preceq v({\E}_{P_1})$. The aim of this section is to prove the following generalization of the second assertion of loc.~cit.
\begin{thm}\label{thmcompare}
In the above context assume that the images of $v({\E}_{P_1})$ and of $v({\E}_{P_2})$ in $\pi_1(M_2)_{\Q}$ coincide. Let $Q=P_1\cap P_2$. Then there is a joint reduction ${\E}_Q\in \Bun_{Q}$ of ${\E}_{P_1}$ and ${\E}_{P_2}$.
\end{thm}
For the proof we need a lemma. Let $W_1, W_2$ be the Weyl groups of $M_1$ and $M_2$, respectively. By ${}^{M_1}W^{M_2}$ we denote the subset of $W$ of elements $w$ which are shortest representatives of their double coset $W_1 wW_2$.

\begin{lemma}\label{lems455}
Let $P_1,P_2$ be standard parabolic subgroups of $G$ with standard Levi factors $M_1,M_2$ and unipotent radicals $N_1,N_2$.
\begin{enumerate}
\item  Let $\nu\in X_*(T)_{\mathbb Q}$ be dominant. For every $w\in W,$ $$\nu\succeq w^{-1}(\nu)$$ in $X_*(T)_{\mathbb Q}$.
\item Let $\nu\in X_*(T)_{\mathbb Q}$ be dominant $P_1$-regular. If $w\in {}^{M_1}W^{M_2}$ is such that the images of $\nu$ and $w^{-1}(\nu)$ in $\pi_1(M_2)_{\Q}$ agree, then $w=1$.
\end{enumerate} 
\end{lemma}
\begin{proof}
(1) is shown in \cite[Lemma 4.8]{S} and follows immediately from the assumption that $\nu$ is dominant. For (2) we replace $\nu$ by a suitable multiple and may thus assume that $\nu\in X_*(T)$, and that the images of $\nu$ and $w^{-1}(\nu)$ in $\pi_1(M_2)$ agree. Let $\nu'$ be the $M_2$-dominant representative in the $W_{M_2}$-orbit of $w^{-1}(\nu)$. Since $\nu,\nu'$ are in the same $W$-orbit and both $M_2$-dominant, there is no root hyperplane for $M_2$ separating the two elements. Therefore, $\nu-\nu'$ is a non-negative linear combination of coroots $\alpha^{\vee}$ for roots $\alpha$ of $T$ in $N_2$. %Wir meinen hier wirklich nur Wurzelhyperebenen, keine affinen, und die Aussage gerade folgt daraus, dass nu' im W-Orbit von nu liegt, und man dann in Spiegelungen an Wurzelhyperebenen zerlegt, und fuer jede den Effekt studiert. Man braucht dann eben keine fuer M zu nehmen.
 Since $\nu=\nu'$ in $\pi_1(M_2)$, this implies that $\nu=\nu'$, hence $w^{-1}(\nu)\in W_{M_2}(\nu)$. Since $w\in{}^{M_1}W^{M_2},$ this implies $w=1$.
\end{proof}

\begin{proof}[Proof of Theorem \ref{thmcompare}]
By \cite[Lemma 4.1]{S}, there is a unique element $w\in {}^{M_1}W^{M_2}$ for which there is an open dense subset $U\subset X$ such that ${\E}_{P_1}|_U$ is in relative position $w$ with respect to ${\E}_{P_2}|_U$.

By the proof of the first part of \cite[Thm.~4.1]{S} we have $$v({\E}_{P_1})\succeq w^{-1}(v({\E}_{P_1}))\succeq v({\E}_{P_2}).$$ By assumption, the images of $v({\E}_{P_1})$ and $v({\E}_{P_2})$ in $\pi_1(M_2)_{\mathbb Q}$ agree. Hence they also coincide with the corresponding image of $w^{-1}(v({\E}_{P_1}))$. By Lemma \ref{lems455}, this implies $w=1$. 

Since $w=1$, the substack 
$$P_1\backslash (P_1\cdot w\cdot P_2)/P_2\hookrightarrow P_1\backslash G/P_2$$ is closed. Thus ${\E}_{P_1}$ and ${\E}_{P_2}$ are in relative position $w=1$ not only generically, but on the entire curve $X$.

Furthermore, $Q=P_1\cap P_2$ is the parabolic subgroup corresponding to the set of simple roots
$$\Delta_Q=\{\alpha_i\in \Delta_{M_1}\mid w^{-1}(\alpha_i)\in \Delta_{M_2}\}=\{\alpha_i\in \Delta_{M_2}\mid w(\alpha_i)\in \Delta_{M_1}\},$$ i.e. equal to both parabolic subgroups $Q_1,Q_2$ of \cite[Cor.~4.1]{S}. Applying the corollary yields the desired reduction of ${\E}$ to $Q$.
\end{proof}

\subsection{The canonical reduction for $G$-bundles on the Fargues-Fontaine curve}\label{seccompffc}

From now on we work use again the context and notation of the main part of the paper. Assume $G$ to be quasi-split. Let $A$ be a maximal split torus, $T$ the centralizer of $A$ and $B$ a Borel subgroup of $G$ containing $T$.

%By uniqueness of the canonical reduction we obtain that it descends to $K$, compare \cite[5.1]{FGtorseurs}. 
For bundles over the Fargues-Fontaine curve, the canonical reduction satisfies the following splitting property. Let ${\E}$ be a $G$-bundle on $X$ and let ${\E}_P$ be its canonical reduction where $P$ is a standard parabolic subgroup of $G$. Let $M$ be the standard Levi factor of $P$. Then by \cite[Prop.~5.16]{FGtorseurs}, ${\E}_P\times^P M$ is a reduction of ${\E}$ to $M$. In particular, ${\E}$ has a semi-stable reduction to a Levi subgroup of $G$.

\begin{thm}\label{thmcompareff}
Let $P_1,P_2$ be standard parabolic subgroups of $G$ with standard Levi subgroups $M_1,M_2$. Let ${\E}_{P_1}\in \Bun_{P_1}^{ss}$ and ${\E}_{P_2}\in \Bun_{P_2}$ be reductions of the same $G$-bundle ${\E}$ on $X$. Assume that ${\E}_{P_1}$ is the canonical reduction of ${\E}$.
\begin{enumerate}
\item $v({\E}_{P_2})\preceq v({\E}_{P_1})$.\end{enumerate}
Assume that the images of $v({\E}_{P_1})$ and of $v({\E}_{P_2})$ in $\pi_1(M_2)_{\Gamma,\Q}$ coincide. Let $Q=P_1\cap P_2$. 
\begin{enumerate}
\item[(2)] There is a joint reduction ${\E}_Q\in \Bun_{Q}$ of ${\E}_{P_1}$ and ${\E}_{P_2}$.
\item[(3)] Let $M$ be the standard Levi subgroup of $Q$. Then ${\E}_Q\times^Q M$ is a reduction of ${\E}$ to $M$.
\end{enumerate}
\end{thm}
\begin{proof}
(1) and (2) are the analogs of the first assertion of \cite[Thm.~4.1]{S}, resp. of the first assertion of Theorem \ref{thmcompare}. The same arguments prove that they still hold in this context.

For the last assertion we use \cite[Lemma 6.3]{CFS} and its proof. Consider the parabolic subgroups $Q\subseteq P_1\subseteq G$ and their Levi subgroups $M\subseteq M_1$. Then $Q\cap M_1$ is a standard parabolic subgroup of $M_1$ with standard Levi subgroup $M$. Since ${\E}_{P_1}$ is the canonical reduction, ${\E}_{P_1}\times^{P_1}M_1$ is a reduction of ${\E}$, and a semi-stable $M_1$-bundle. By the first part of the proof of \cite[Lemma 6.3]{CFS}, the reduction ${\E}_Q$ corresponds to a unique reduction $({\E}_{P_1}\times^{P_1}M_1)_{M_1\cap Q}$ of ${\E}_{P_1}\times^{P_1}M_1$ to $M_1\cap Q$.  Since ${\E}_{Q}$ is a reduction of ${\E}_{P_1}$, we have that $v({\E}_{P_1})-v({\E}_Q)$ is a linear combination of coroots in $M_1$. Since ${\E}_Q$ is a reduction of ${\E}_{P_2}$, the image of $v({\E}_Q)$ in $\pi_1(M_2)_{\Q}$ coincides with the image of $v({\E}_{P_2})$. By the assumption of the theorem this agrees with the image of $v({\E}_{P_1})$. Thus $v({\E}_{P_1})-v({\E}_Q)$ is in fact a linear combination of coroots in $M_1\cap M_2=M_Q$. Hence the assumption of \cite[Lemma 6.3]{CFS} is satisfied for the $M_1$-bundle ${\E}_{P_1}\times^{P_1}M_1$, and the lemma finishes the proof.
\end{proof}

\bibliography{Bib}

\begin{thebibliography}{GHN19}

\bibitem[And09]{Andre}
Yves Andr\'{e}.
\newblock Slope filtrations.
\newblock {\em Confluentes Math.}, 1(1):1--85, 2009.

\bibitem[Beh95]{B}
Kai~A. Behrend.
\newblock Semi-stability of reductive group schemes over curves.
\newblock {\em Math. Ann.}, 301(2):281--305, 1995.

\bibitem[BH04]{BH}
Indranil Biswas and Yogish~I. Holla.
\newblock Harder-{N}arasimhan reduction of a principal bundle.
\newblock {\em Nagoya Math. J.}, 174:201--223, 2004.

\bibitem[BL95]{BL}
Arnaud Beauville and Yves Laszlo.
\newblock Un lemme de descente.
\newblock {\em C. R. Acad. Sci. Paris S\'{e}r. I Math.}, 320(3):335--340, 1995.

\bibitem[CFS21]{CFS}
Miaofen Chen, Laurent Fargues, and Xu~Shen.
\newblock On the structure of some $p$-adic period domains, 2021.

\bibitem[Che19]{MC}
Miaofen Chen.
\newblock Fargues-{R}apoport conjecture for $p$-adic period domains in the
  non-basic case, 2019.

\bibitem[Cor18]{Cornut}
Christophe Cornut.
\newblock On {H}arder-{N}arasimhan filtrations and their compatibility with
  tensor products.
\newblock {\em Confluentes Math.}, 10(2):3--49, 2018.

\bibitem[CPI19]{CP}
Christophe Cornut and Macarena Peche~Irissarry.
\newblock Harder-{N}arasimhan filtrations for {B}reuil-{K}isin-{F}argues
  modules.
\newblock {\em Ann. H. Lebesgue}, 2:415--480, 2019.

\bibitem[DOR10]{DOR}
Jean-Fran\c{c}ois Dat, Sascha Orlik, and Michael Rapoport.
\newblock {\em Period domains over finite and {$p$}-adic fields}, volume 183 of
  {\em Cambridge Tracts in Mathematics}.
\newblock Cambridge University Press, Cambridge, 2010.

\bibitem[Fal10]{Falt2010}
Gerd Faltings.
\newblock Coverings of {$p$}-adic period domains.
\newblock {\em J. Reine Angew. Math.}, 643:111--139, 2010.

\bibitem[Far]{FGtorseurs}
Laurent Fargues.
\newblock $g$-torseurs en th\'eorie de hodge $p$-adique.
\newblock {\em to appear in Comp. Math.}

\bibitem[Far20]{FarGtors}
Laurent Fargues.
\newblock {$G$}-torseurs en th\'{e}orie de {H}odge {$p$}-adique.
\newblock {\em Compos. Math.}, 156(10):2076--2110, 2020.

\bibitem[Farnt]{Far19}
Laurent Fargues.
\newblock Th\'{e}orie de la r\'{e}duction pour les groupes $p$-divisibles.
\newblock preprint.

\bibitem[FF18]{FF}
Laurent Fargues and Jean-Marc Fontaine.
\newblock Courbes et fibr\'{e}s vectoriels en th\'{e}orie de {H}odge
  {$p$}-adique.
\newblock {\em Ast\'{e}risque}, (406):xiii+382, 2018.
\newblock With a preface by Pierre Colmez.

\bibitem[FS21]{FarguesScholze}
Laurent Fargues and Peter Scholze.
\newblock Geometrization of the local {L}anglands correspondence, preprint,
  2021.

\bibitem[GHN19]{GHN}
Ulrich G\"{o}rtz, Xuhua He, and Sian Nie.
\newblock Fully {H}odge-{N}ewton decomposable {S}himura varieties.
\newblock {\em Peking Math. J.}, 2(2):99--154, 2019.

\bibitem[Har11]{HartlAnnals}
Urs Hartl.
\newblock Period spaces for {H}odge structures in equal characteristic.
\newblock {\em Ann. of Math. (2)}, 173(3):1241--1358, 2011.

\bibitem[HN75]{HN}
G.~Harder and M.~S. Narasimhan.
\newblock On the cohomology groups of moduli spaces of vector bundles on
  curves.
\newblock {\em Math. Ann.}, 212:215--248, 1974/75.

\bibitem[KL15]{KL}
Kiran~S. Kedlaya and Ruochuan Liu.
\newblock Relative {$p$}-adic {H}odge theory: foundations.
\newblock {\em Ast\'{e}risque}, (371):239, 2015.

\bibitem[Kot03]{Kot03}
Robert~E. Kottwitz.
\newblock On the {H}odge-{N}ewton decomposition for split groups.
\newblock {\em Int. Math. Res. Not.}, (26):1433--1447, 2003.

\bibitem[MV07]{MirkovicVilonen}
Ivan Mirkovi\'{c} and Kari Vilonen.
\newblock Geometric {L}anglands duality and representations of algebraic groups
  over commutative rings.
\newblock {\em Ann. of Math. (2)}, 166(1):95--143, 2007.

\bibitem[Orl06]{Orlik06}
Sascha Orlik.
\newblock On {H}arder-{N}arasimhan strata in flag manifolds.
\newblock {\em Math. Z.}, 252:209--222, 2006.

\bibitem[Rap18]{RapoportAppendix}
Michael Rapoport.
\newblock Accessible and weakly accessible period domains.
\newblock {\em {A}ppendix to {\it On the $p$-adic cohomology of the Lubin-Tate
  tower} by P.~Scholze, Ann. Sci. ENS}, 51:856--863, 2018.

\bibitem[RZ96]{RZ96}
Michael Rapoport and Thomas Zink.
\newblock {\em Period spaces for {$p$}-divisible groups}, volume 141 of {\em
  Annals of Mathematics Studies}.
\newblock Princeton University Press, Princeton, NJ, 1996.

\bibitem[Sch15]{S}
Simon Schieder.
\newblock The {H}arder-{N}arasimhan stratification of the moduli stack of
  {$G$}-bundles via {D}rinfeld's compactifications.
\newblock {\em Selecta Math. (N.S.)}, 21(3):763--831, 2015.

\bibitem[She19]{Shen}
Xu~Shen.
\newblock Harder-{N}arasimhan strata and $p$-adic period domains, 2019.

\bibitem[SR72]{Saavedra}
Neantro Saavedra~Rivano.
\newblock Cat\'egories {T}annakiennes.
\newblock 265:418, 1972.

\bibitem[SW20]{SW17}
Peter Scholze and Jared Weinstein.
\newblock Berkeley lectures on $p$-adic geometry.
\newblock 2020.

\bibitem[Vie21]{Viehmann20}
Eva Viehmann.
\newblock On {N}ewton strata in the ${B}_{{\rm dR}}^+$-{G}rassmannian,
  preprint, 2021.

\end{thebibliography}
\bibliographystyle{alpha}
\end{document}